\documentclass{article}

\usepackage[utf8]{inputenc} % allow utf-8 input
\usepackage[T1]{fontenc}    % use 8-bit T1 fonts
\usepackage{hyperref}       % hyperlinks
\usepackage{url}            % simple URL typesetting
\usepackage{booktabs}       % professional-quality tables
\usepackage{amsfonts}       % blackboard math symbols
\usepackage{nicefrac}       % compact symbols for 1/2, etc.
\usepackage{microtype}      % microtypography
\usepackage{subfigure}
\usepackage{graphicx}
\usepackage{wrapfig}
\graphicspath{ {./images/} }
\usepackage[toc]{appendix}
\usepackage{hyperref,amsmath,amssymb,float,amsthm,xcolor,bbm,enumitem,bm}
\usepackage{geometry}
\usepackage[ruled,vlined]{algorithm2e}
\usepackage{algorithmic}

\newtheorem{theorem}{Theorem}
\newtheorem{definition}{Definition}
\newtheorem{remark}{Remark}
\newtheorem{lemma}{Lemma}
\newtheorem{assumption}{Assumption}
\newtheorem{corollary}{Corollary}

\DeclareMathOperator{\VEC}{Vec}

\title{Improved Rate of First Order Algorithms for Entropic Optimal Transport}

\author{
Yiling Luo, Yiling Xie and Xiaoming Huo\\
School of Industrial and Systems Engineering\\
Georgia Institute of Technology\\
\texttt{\{yluo373, yxie350, huo\}@gatech.edu}
}

\begin{document}

\maketitle

% The \author macro works with any number of authors. There are two commands
% used to separate the names and addresses of multiple authors: \And and \AND.
%
% Using \And between authors leaves it to LaTeX to determine where to break the
% lines. Using \AND forces a line break at that point. So, if LaTeX puts 3 of 4
% authors names on the first line, and the last on the second line, try using
% \AND instead of \And before the third author name.

\begin{abstract}
This paper improves the state-of-the-art rate of a first-order algorithm for solving entropy regularized optimal transport. 
The resulting rate for approximating the optimal transport (OT) has been improved from $\widetilde{\mathcal{O}}({n^{2.5}}/{\epsilon})$ to $\widetilde{\mathcal{O}}({n^2}/{\epsilon})$, where $n$ is the problem size and $\epsilon$ is the accuracy level.  
In particular, we propose an accelerated primal-dual stochastic mirror descent algorithm with variance reduction. 
Such special design helps us improve the rate compared to other accelerated primal-dual algorithms. 
We further propose a batch version of our stochastic algorithm, which improves the computational performance through parallel computing.
To compare, we prove that the computational complexity of the Stochastic Sinkhorn algorithm is $\widetilde{\mathcal{O}}({n^2}/{\epsilon^2})$, which is slower than our accelerated primal-dual stochastic mirror algorithm. 
Experiments are done using synthetic and real data, and the results match our theoretical rates.
Our algorithm may inspire more research to develop accelerated primal-dual algorithms that have rate $\widetilde{\mathcal{O}}({n^2}/{\epsilon})$ for solving OT. 
\end{abstract} 

\section{Introduction}
The \textit{Optimal Transport} (OT) \cite{monge1781memoire, kantorovich1942translocation, villani2009optimal} is an optimization problem that has been actively studied. 
In this section, we review the OT problem. 
In Section \ref{sec:1.1}, we review the OT formulation and its related concepts. 
In Section \ref{sec:1.1b}, we survey the existing algorithms for solving OT and summarize our contribution given the literature background. 

\subsection{Optimal Transport} \label{sec:1.1}
We review the definition of OT.
Given a cost matrix $C \in \mathbb{R}_+^{n\times n}$ and two vectors $\bm{p},\bm{q}\in \Delta_n$, where $\Delta_n := \{\bm{a}\in \mathbb{R}^n_+ : \bm{a}^T \bm{1} = 1\}$ is the standard simplex, OT is defined as follows:
\begin{equation}\label{eq:OT_nopen}
    \min_{X \in \mathcal{U}(\bm{p},\bm{q})} \langle C, X \rangle,
\end{equation}
where $\mathcal{U}(\bm{p},\bm{q}) := \left\{X\in \mathbb{R}^{n\times n}_{+}\left|X\bm{1} = \bm{p}, X^T\bm{1} = \bm{q}\right.\right\}$, and $\langle C, X \rangle := \sum_{i,j=1}^n C_{i,j} X_{i,j}$. 

The \textit{$\epsilon$-solution} is always used when evaluating algorithm efficiency for solving OT, so we review its definition as follows. 
Denote the optimal solution of problem \eqref{eq:OT_nopen} as $X^*$, an $\epsilon-$solution $\widehat{X}$ is such that:
\begin{align*}
    \begin{split}
        &\widehat{X}\in \mathcal{U}(\bm{p},\bm{q});\\
        & \langle C,\widehat{X}\rangle \leq \langle C,{X}^*\rangle + \epsilon.
    \end{split}
\end{align*}
Note that for a stochastic algorithm, the second condition is replaced by $\mathbbm{E}\langle C,\widehat{X}\rangle \leq \langle C,{X}^*\rangle + \epsilon$. 

Our paper adopts a two-step approach \cite{altschuler2017near} for finding an $\epsilon$-solution to problem \eqref{eq:OT_nopen}. 
In the first step, one finds an approximate solution $\widetilde{X}$ to the \textit{entropic OT} problem \eqref{eq:OT_pen}.
\begin{equation}\label{eq:OT_pen}
    \min_{X \in \mathcal{U}(\bm{p}^\prime, \bm{q}^\prime)} \langle C, X \rangle - \eta H(X),
\end{equation}
where $H(X) = -\sum_{i,j} X_{i,j} \log (X_{i,j})$ is the entropy. 
In the second step, one rounds $\widetilde{X}$ to the original feasible region $\mathcal{U}(\bm{p},\bm{q})$. 
By taking proper parameters $\eta,\bm{p}^\prime, \bm{q}^\prime$ and requiring a suitable accuracy level when approximating problem \eqref{eq:OT_pen}, the work \cite{altschuler2017near} guarantees the final solution to be an $\epsilon$-solution to problem \eqref{eq:OT_nopen}. 

%The computation in the two-step approach above is dominated by the first step, thus a lot of algorithms are proposed to solve \eqref{eq:OT_pen} efficiently. 

\begin{table*}[ht]
\caption{%Order of complexity of OT algorithms. 
In this table, we list the year of the relevant publication, the names of the methods, the simplified version of its computational complexity, and whether (a $\surd$ sign) or not (an $\times$ sign) the method solves entropic OT as an intermediate step for approximating OT in columns. The mark of ``(This Paper)'' indicates a rate derived in this paper. It is clear that our method achieves the lowest rate among the methods that solve entropic OT.}
\label{tab:01}
\begin{center}
\begin{small}
\begin{sc}
%\vskip -0.15in
\begin{tabular}{lccc}
\toprule
Year &  Algorithm& Order of Complexity & Solves Entropic OT \\
\midrule
2013 & Sinkhorn \cite{cuturi2013sinkhorn} & $n^{2}/\epsilon^2$ \cite{dvurechensky2018computational} & $\surd$ \\
2017 & Greenkhorn \cite{altschuler2017near} & $n^{2}/\epsilon^3$ \cite{altschuler2017near}; $n^{2}/\epsilon^2$ \cite{lin2019efficient} & $\surd$ \\
2018 & Stochastic Sinkhorn \cite{abid2018stochastic} & $n^{2}/\epsilon^3$; $n^{2}/\epsilon^2$ (This paper) & $\surd$\\
2018 & APDAGD \cite{dvurechensky2018computational}& $n^{2.5}/\epsilon$ & $\surd$ \\
2018 & Packing LP \cite{blanchet2018towards,quanrud2018approximating} & $n^2/\epsilon$ &$\times$\\
2018 & Box Constrained Newton \cite{blanchet2018towards} & $n^2/\epsilon$ & $\surd$\\ 
2019 & APDAMD \cite{lin2019efficient} & $n^{2.5}/\epsilon$ &$\surd$\\
2019 & Dual Extrapolation \cite{jambulapati2019direct} & $n^2/\epsilon$ & $\times$\\
2019 & Accelerated Sinkhorn \cite{lin2022efficiency} & $n^{7/3}/\epsilon^{4/3}$ & $\surd$\\
2019 & Dijkstra’s search + DFS \cite{lahn2019graph} & $n^2/\epsilon + n/\epsilon^2$ & $\times$\\
2020 & APDRCD \cite{guo2020fast} & $n^{2.5}/\epsilon$ & $\surd$\\
2021 & AAM \cite{guminov2021combination}& $n^{2.5}/\epsilon$ & $\surd$ \\
2022 & Hybrid Primal-Dual \cite{chambolle2022accelerated} & $ n^{2.5}/\epsilon$ &$\surd$\\
2022 & PDASGD \cite{xie2022} & $ n^{2.5}/\epsilon$ &$\surd$\\
2022 & PDASMD & $\bm{n^{2}/\epsilon}$  (This paper) &$\surd$\\
\bottomrule
\end{tabular}
\end{sc}
\end{small}
\end{center}
\vskip -0.15in
\end{table*}

\subsection{Literature Review}\label{sec:1.1b}
We review the state-of-the-art algorithms that solve OT by the two-step approach and summarize their computational complexity (measured by the number of numerical operations) for giving an $\epsilon$-solution to OT in Table \ref{tab:01}. 
The computational complexities in Table \ref{tab:01} are shown in their order of $n$ and $\epsilon$, where the $\log(n)$ term is omitted. 

There are four main techniques to solve problem \eqref{eq:OT_pen} in current literature:
\begin{itemize}[leftmargin=*,noitemsep]
\item The first technique solves the dual problem of problem \eqref{eq:OT_pen} by the Bregman projection technique. 
Specifically, this technique partitions the dual variables into blocks and iteratively updates each block. 
Algorithms that use this technique include the Sinkhorn algorithm \cite{cuturi2013sinkhorn}, the Greenkhorn algorithm \cite{altschuler2017near}, and the Stochastic Sinkhorn algorithm \cite{abid2018stochastic}. 

\item The second technique also solves the dual problem of problem \eqref{eq:OT_pen} but uses accelerated first-order methods. 
Algorithms that use this technique include accelerated gradient descent (APDAGD) \cite{dvurechensky2018computational}, accelerated mirror descent (APDAMD) \cite{lin2019efficient}, accelerated alternating minimization (AAM) \cite{guminov2021combination}, accelerated randomized coordinate descent (APDRCD) \cite{guo2020fast} and accelerated stochastic gradient descent (PDASGD) \cite{xie2022}. 
This technique can also be combined with the first technique. See, for example, the accelerated Sinkhorn algorithm in reference \cite{lin2022efficiency}. 

\item The third technique solves the dual problem of problem \eqref{eq:OT_pen} by second-order algorithms. 
An instance that uses this technique is the box-constrained Newton algorithm \cite{blanchet2018towards}. 

\item The fourth technique minimizes the primal-dual gap of problem \eqref{eq:OT_pen}. 
An instance that uses this technique is the hybrid primal-dual algorithm \cite{chambolle2022accelerated}. 
\end{itemize}

Besides works that use the two-step approach to solve the entropic OT first, some works directly solve the unpenalized OT problem \eqref{eq:OT_nopen} by linear programming \cite{blanchet2018towards,quanrud2018approximating}, dual-extrapolation \cite{jambulapati2019direct}, or graph-based search algorithm \cite{lahn2019graph}.

We compare the computational complexity in Table \ref{tab:01} of our algorithm with other state-of-the-art algorithms as follows. 

First, our PDASMD algorithm belongs to the second class of algorithms to solve the entropic OT problem \eqref{eq:OT_pen}. 
All other algorithms in this class reported a rate of $\widetilde{\mathcal{O}}(n^{2.5}/\epsilon)$ for approximating OT, while our algorithm has a better rate of $\widetilde{\mathcal{O}}(n^{2}/\epsilon)$. 
Thus our algorithm improves the rate for this class. 
The advantage of our algorithm mainly comes from the special technique that we use: though all the algorithms in this class use the acceleration technique, no accelerated variance reduction version of stochastic mirror descent has been tried in the previous algorithms. 
We apply those techniques to entropic OT and find that they lead to a better theoretical rate. 

Second, our PDASMD algorithm still reports the best rate among all algorithms for solving entropic OT. 
There is only one algorithm on entropic OT that achieved the same rate: the box-constrained Newton algorithm. 
However, we note that the Newton algorithm is a second-order algorithm, which requires computing the Hessian of the objective function. 
By its second-order nature, each step of the Newton algorithm will be expensive in terms of computation and memory. 
On the other hand, our PDASMD algorithm is based on mirror descent, which is a first-order algorithm. 
Our PDASMD algorithm is thus easier to implement.   

Finally, the algorithms that directly solve the original OT problem also report the same optimal rate as our PDASMD algorithm, including the packing LP algorithm, the dual extrapolation algorithm, and the graph-based Dijkstra DFS algorithm (when $\epsilon \gtrsim 1/n$). 
Compared with those algorithms, we have the extra advantage that our algorithm can not only approximate the OT problem but also solve the entropic OT. 
Thus, when one wants to solve the entropic OT, our algorithm is still preferred. 

\paragraph{Our Contribution} 
We summarize two main contributions in this work as follows.

\begin{itemize}[leftmargin=*,noitemsep]
    \item We propose an accelerated primal-dual stochastic algorithm that has computational complexity $\widetilde{\mathcal{O}}(n^2/\epsilon)$ for solving OT. 
    Every step of our algorithm is defined by simple arithmetic operations and is counted in the complexity calculation. 
    Thus our algorithm is practical. 
    Moreover, compared with other algorithms that achieve the same rate for solving OT: our algorithm has the extra advantage that it can also be applied to entropic OT; it is a first-order algorithm, so it can be easily implemented without computing the Hessian. 
    We also propose a batch version of our algorithm to increase the computational power.
    
    \item We prove that the computational complexity of the Stochastic Sinkhorn algorithm is $\widetilde{\mathcal{O}}({n^2}/{\epsilon^2})$, instead of the $\widetilde{\mathcal{O}}({n^2}/{\epsilon^3})$ rate in the literature. 
    Our proved rate for Stochastic Sinkhorn matches the state-of-the-art rate of Sinkhorn and Greenkhorn. Moreover, the provable rate by our accelerated primal-dual stochastic algorithm is better than that of the Stochastic Sinkhorn, which again illustrates the advantage of our algorithm. 
\end{itemize}

\paragraph{Paper Organization}
The rest of the paper is organized as follows. 
In Section \ref{sec:2}, we present our main algorithm of Primal-Dual Accelerated Stochastic Proximal Mirror Descent (PDASMD), show its convergence, and analyze its complexity for solving OT;
as a comparison, we also prove the rate of Stochastic Sinkhorn, which is improved over the existing result. 
In Section \ref{sec:3}, we develop a batch version of PDASMD and show its convergence and computational complexity. 
We run numerical examples in Section \ref{sec:5} to support our theorems. 
In Section \ref{sec:6}, we discuss the findings of this work and some future research.

\section{Primal-Dual Accelerated Stochastic Proximal Mirror Descent (PDASMD)}\label{sec:2}
In this section, we present our PDASMD algorithm for solving a linear constrained convex problem, which includes the entropic OT as a special case. 
We analyze the convergence rate of the PDASMD algorithm, then apply it to OT and derive the computational complexity. 
As a comparison, we also analyze the computational complexity of the Stochastic Sinkhorn. 
Since our algorithm uses the Proximal Mirror Descent technique, we review the background of such a technique in Appendix \ref{app:PMD} and briefly explain why it is suitable for entropic OT. 

\subsection{Definition and Notation}

We first introduce some notations that we will use throughout the rest of this paper.

\textbf{Notations}:
For a vector $\bm{a}$: 
let $sign(\bm{a})$ be such that $(sign(\bm{a}))_i = 1 $ if $a_i > 0$ and $-1$ otherwise. 
Let $\mathbf{1}_n$ be the $n$-dimensional vector where each element is $1$. 
For matrices $X \in \mathbb{R}^{n\times o}, Y \in \mathbb{R} ^{p\times q}$: 
let $ X\otimes Y$ denote the standard Kronecker product; let $\exp(X)$ and $\log(X)$ be the element-wise exponential and logarithm of $X$; 
let $\|X\|_2$ be the operator norm of $X$ and $\|X\|_{\infty}$ be $\max_{i,j} |X_{i,j}|$; 
denote the matrix norm induced by two arbitrary vector norms $\|\cdot\|_H$ and $\|\cdot\|_E$ as $\|X\|_{E\to H} := \max_{\bm{a}:\|\bm{a}\|_E\leq 1} \|X\bm{a}\|_H$; 
denote the vectorization of $X$ as $\VEC(X)=(X_{11},...,X_{n1},X_{12},...,X_{n2},...,X_{1o},...,X_{no})^{T}$. 
For two non-negative real values $s(\kappa)$ and $t(\kappa)$, denote $s(\kappa)=\Theta(t(\kappa))$ if $\exists k > 0$ and $K > 0$ such that $k t(\kappa) \leq s(\kappa)\leq K t(\kappa)$; 
denote $s(\kappa)=\mathcal{O}(t(\kappa))$ if $\exists K > 0$ such that $s(\kappa)\leq K t(\kappa)$; denote $s(\kappa)=\widetilde{\mathcal{O}}(t(\kappa))$ to indicate the previous inequality where $K$ depends on some logarithmic function of $\kappa$. 

Next, we review some key definitions that will be useful. \footnote{Our definitions follow those in \cite{allen2017katyusha}.}    
\begin{definition}[Strong convexity]
$f: \mathcal{Q}\to \mathbb{R}$ is $\alpha$-strongly convex w.r.t. $\|\cdot\|_H$ if $\:\forall \bm{x},\bm{y}\in \mathcal{Q}$:
$$f(\bm{y})\geq f(\bm{x}) + \langle \nabla f(\bm{x}), \bm{y} - \bm{x}\rangle + \frac{\alpha}{2}\|\bm{x}-\bm{y}\|_H^2.$$
\end{definition}

\begin{definition}[Smoothness]
A convex function $f: \mathcal{Q}\to \mathbb{R}$ is $\beta$-smooth w.r.t. $\|\cdot\|_H$ if $\:\forall \bm{x},\bm{y}\in \mathcal{Q}$:
$$\|\nabla f(\bm{x}) - \nabla f(\bm{y})\|_{H,*}\leq\beta\|\bm{x}-\bm{y}\|_H,$$
where $\|\bm{u}\|_{H,*} := \max_{\bm{v}}\{\langle \bm{u},\bm{v}\rangle : \|\bm{v}\|_H \leq 1\} $ is the dual norm of $\|\cdot\|_H$. Or equivalently,
$$f(\bm{y})\leq f(\bm{x}) + \langle \nabla f(\bm{x}),\bm{y}-\bm{x}\rangle + \frac{\beta}{2}\|\bm{x}-\bm{y}\|^2_H.$$
\end{definition}

\begin{definition}[Bregman divergence]
For a mirror function $w(\cdot)$ that is $1$-strongly convex w.r.t. $\|\cdot\|_H$, we denote by $V_{\bm{x}}(\bm{y})$ the Bregman divergence w.r.t. $\|\cdot\|_H$ generated by $w(\cdot)$, where
$$V_{\bm{x}}(\bm{y}) := w(\bm{y}) - w(\bm{x}) - \langle \nabla w(\bm{x}), \bm{y} - \bm{x}\rangle.$$
One can conclude from the definition that
$$V_{\bm{x}}(\bm{y}) \geq \frac{1}{2}\|\bm{x} - \bm{y}\|_H^2.$$
If we further assume that the mirror function $w(\cdot)$ is $\gamma$-smooth w.r.t. $\|\cdot\|_H$, we then have
$$V_{\bm{x}}(\bm{y}) \leq \frac{\gamma}{2}\|\bm{x} - \bm{y}\|_H^2.$$
\end{definition}

\subsection{General Formulation and PDASMD Algorithm}\label{sec:2.2}
In this section, we first state a general linear constrained problem and explain how it includes entropic OT as a special case. 
We then propose our algorithm to solve this general problem.
Finally, we show the convergence rate of our algorithm.

We consider a linear constrained problem as follows:
\begin{align}
\begin{split}\label{eq:obj}
    \min_{\bm{x}\in \mathbb{R}^m} f(\bm{x})\quad\quad  s.t. \quad A\bm{x} = \bm{b} \in \mathbb{R}^l,
\end{split}
\end{align}
where $f$ is strongly convex. 
One observes that the entropic OT \eqref{eq:OT_pen} is a special case of problem \eqref{eq:obj} with $\bm{x} = \VEC(X)$, $f(\bm{x}) = \langle \VEC(C), \bm{x} \rangle + \eta \sum_{i,j = 1}^n x_{in + j}\log(x_{in + j})$, $\bm{b} = (\bm{p}^T, \bm{q}^T)^T$, $A = \left[\begin{array}{c}\bm{1}^T\otimes I_n\\I_n \otimes \bm{1}^T \end{array}\right]$.

A standard approach for solving the constrained problem \eqref{eq:obj} is to optimize its Lagrange dual problem \eqref{eq:dual}:
\begin{align}\label{eq:dual}
\min_{\bm{\lambda}}\{\phi(\bm{\lambda}) :=& \langle \bm{\lambda} ,\bm{b} \rangle + \max_{\bm{x}}( -f(\bm{x}) - \langle A^T \bm{\lambda} ,\bm{x}\rangle)\nonumber\\
=& \langle \bm{\lambda} ,\bm{b} \rangle -f(\bm{x}(\bm{\lambda})) - \langle A^T \bm{\lambda} , \bm{x}(\bm{\lambda})\rangle\},
\end{align}
where by F.O.C. $\bm{x}(\bm{\lambda})$ is such that
\begin{align}\label{eq:pd_relation}
\nabla_{\bm{x}} f(\bm{x}(\bm{\lambda})) = -A^T\bm{\lambda}.
\end{align}
Since problem \eqref{eq:obj} is a linear constrained convex problem, the strong duality holds. 
Thus solving problem \eqref{eq:obj} is equivalent to solving its dual problem \eqref{eq:dual}. 
In particular, we develop a stochastic algorithm for the case that the dual is of finite sum form.
We further assume that all terms in the finite sum are smooth for convergence analysis.
The conditions on the dual are formalized as follows:
\begin{assumption}[Finite-sum dual]\label{ass:dual}
Assume that the dual can be written as $\phi(\bm{\lambda}) = \frac{1}{m} \sum_{i= 1}^m \phi_i(\bm{\lambda})$, where $\phi_i$ is convex and $L_i-$Lipchitz smooth w.r.t. an arbitrary $\|\cdot\|_{H}$ norm. 
\end{assumption}
Note that the assumption on the dual is reasonable and can be satisfied by some problems, including entropic OT.
We now give a concrete example that the assumption holds. 
Consider a primal objective $f(\bm{x}) = \sum_{i=1}^m f_i(x_i)$ where each $f_i$ is $\nu-$strongly convex w.r.t. another arbitrary norm $\|\cdot\|_{E}$ (note that it can be different from the $\|\cdot\|_H$ norm). 
In this case, we can solve the primal-dual relationship in equation \eqref{eq:pd_relation} to get:
\begin{align*}
x_i(\bm{\lambda}) = (\nabla f_i)^{-1}(-\bm{a}_i^T\bm{\lambda}), i = 1,\ldots,m,
\end{align*}
where $\bm{a}_i$ is the $i$th column of $A$. 
As a consequence, the dual problem \eqref{eq:dual} can be written as a finite sum:
\begin{align*}
\begin{split}
   \phi(\bm{\lambda})   &= \frac{1}{m}\sum_{i=1}^m (\langle \bm{\lambda} ,\bm{b}_i \rangle -m f_i(x_i(\bm{\lambda})) - m \bm{a}_i^T \bm{\lambda} x_i(\bm{\lambda}))\\
   & := \frac{1}{m}\sum_{i=1}^m \phi_i(\bm{\lambda}),
\end{split}
\end{align*}
where $\bm{b}_i$'s are arbitrarily chosen vectors satisfying the constraint $\sum_{i=1}^m \bm{b}_i = m\bm{b}$. 
One can check that $\nabla \phi_i (\bm{\lambda}) = \bm{b}_i - m x_i(\bm{\lambda})\bm{a}_i$. 
By \cite{nesterov2005smooth}, $\phi_i$ is convex and $L_i-$Lipchitz smooth w.r.t. $\|\cdot\|_{H}$ norm, where $L_i \leq \frac{m}{\nu}\|\bm{a}_i\|_{E\to H,*}$.

With the finite sum representation of $\phi$, we propose a PDASMD algorithm (Algorithm \ref{algo:PDSMD}) to solve problem \eqref{eq:obj}. 
We add a few remarks to explain the algorithm as follows.
\begin{algorithm}[h]
   \caption{Primal-Dual Accelerated Stochastic Proximal Mirror Descent (PDASMD)}
   \begin{algorithmic}[1]
   \label{algo:PDSMD}
   \STATE Initialize $l$ the number of inner iterations; $\tau_2= \frac{1}{2}$, $\bm{y}_0 = \bm{z}_0 = \widetilde{\bm{v}}_0=\bm{v}_0 = \bm{0}$, $C_0=D_0= 0$; 
   choose a mirror function $w(\cdot)$ that is 1-strongly convex and $\gamma$-smooth w.r.t. $\|\cdot\|_{H}$, and denote by $V_{\bm{x}}(\bm{y})$ the Bregman divergence generated by $w(\cdot)$; 
   take $\bar{L} = (\sum_{i=1}^m L_i)/m$, where $L_i$ is the smoothness (w.r.t. $\|\cdot\|_H$) for each component $\phi_i$ of the dual function $\phi(\cdot)$ in Assumption \ref{ass:dual}. 
   \FOR{s = 0,\ldots,S-1}
   \STATE $\tau_{1,s}\leftarrow{2}/{(s+4)}$; $\alpha_s\leftarrow {1}/{(9 \tau_{1,s}\bar{L})}$;
   \STATE $\bm{\mu}^{s}\leftarrow \nabla \phi(\widetilde{\bm{v}}^s)$.
   \FOR{$j=0$ to $l-1$}
   \STATE $k\leftarrow (sl)+j$;
   \STATE $\bm{v}_{k+1}\leftarrow\tau_{1,s}\bm{z}_k+\tau_2\widetilde{\bm{v}}^{s}+(1-\tau_{1,s}-\tau_2)\bm{y}_k$;
   \STATE Pick i randomly from $\{1, 2, \ldots, m\}$, each with probability $p_i := L_i/m\bar{L}$;
   \STATE $\widetilde{\bm{\nabla}}_{k+1}\leftarrow \bm{\mu}^{s}+\frac{1}{mp_i}(\nabla \phi_{i}(\bm{v}_{k+1})-\nabla \phi_{i}(\widetilde{\bm{v}}^{s}))$;
   \STATE $\bm{z}_{k+1}=\arg\min_{\bm{z}}\{\frac{1}{\alpha_s}V_{ \bm{z}_{k}}(\bm{z})+\langle\widetilde{\bm{\nabla}}_{k+1},\bm{z}\rangle\}$;
   \STATE $\bm{y}_{k+1}=\arg\min_{\bm{y}}\{ \frac{9\bar{L}}{2}\|\bm{y}- \bm{v}_{k+1}\|_{H}^2+\langle\widetilde{\bm{\nabla}}_{k+1},\bm{y}\rangle \}$.
   \ENDFOR
   \STATE $\widetilde{\bm{v}}^{s+1}\leftarrow \frac{1}{l}\sum\limits_{j=1}^{l}\bm{y}_{sl+j}$;
   \STATE $C_{s} \leftarrow C_s+\frac{1}{\tau_{1,s}}$;
   \STATE Pick $\widetilde{\bm{y}}_{s}$ uniform randomly from $\{\bm{y}_{sl+j}\}_{j=1}^l$, update $D_s \leftarrow D_s+\frac{1}{\tau_{1,s}} \bm{x}(\widetilde{\bm{y}}_{s})$, where $\bm{x}(\cdot)$ is given by equation \eqref{eq:pd_relation};
   \STATE $\bm{x}^s={D_s}/{C_s}$.
   \ENDFOR
   \STATE \textbf{Output:} $\widetilde{\bm{x}} = \bm{x}^{S-1}$.
   \end{algorithmic}
\end{algorithm}
\begin{remark}
To run the algorithm, one should choose a specific $\|\cdot\|_H$ norm and a mirror function $w(\cdot)$.
Those choices have a direct impact on the mirror descent step 10 and proximal gradient descent step 11:
if we let $\|\cdot\|_H = \|\cdot\|_2$ and $w(\cdot) = \frac{1}{2}\|\cdot\|_2^2$, both steps reduce to stochastic gradient descent, then the algorithm essentially reduces to the PDASGD algorithm in \cite{xie2022}.
\end{remark}

\begin{remark}
The primal variables $\bm{x}$'s in Algorithm \ref{algo:PDSMD} are updated by Steps 14 through 16, and we explain those steps as follows:
The iterates in Steps 14 through 16 essentially leads to $\bm{x}^{S-1}=\left(\sum\limits_{s=0}^{S-1}\bm{x}(\widetilde{\bm{y}}_{s})/\tau_{1,s}\right)\left/\left(\sum\limits_{s=0}^{S-1}(1/\tau_{1,s})\right)\right.$. 
We express such updates in $\bm{x}^s$ in an iterative way to avoid storing all updates of $\widetilde{\bm{y}}_{s}$'s. 
In this way, our algorithm is memory efficient. 
\end{remark}

\begin{remark}
The dual variables $\bm{v},\bm{z},\bm{y}$'s are updated by Steps 2 through 13. 
The update consists of outer loops indexed by $s$ and inner loops indexed by $j$, which uses the variance reduction and acceleration technique in \cite{allen2017katyusha} (Algorithm 5 in that paper). 
We now summarize the variance reduction and acceleration technique for a better understanding of our algorithm.
\\
The \textbf{variance reduction} in Algorithm \ref{algo:PDSMD} is step 9, which works as follows:
For the finite-sum dual $\phi(v) = \frac{1}{m} \sum_{i= 1}^m \phi_i(v)$, 
a stochastic algorithm without variance reduction updates the parameter estimation using $\nabla \phi_i(v)$, which in general has $Var[\nabla \phi_i(v)]\neq 0, \forall v$ and thus needs the step size $\to 0$ for convergence. 
A variance reduced algorithm replaces $\nabla \phi_{i}(v)$ by $A_k = \nabla \phi_i(v) - B_k + \mathbbm{E}[B_k]$. 
When $B_t$ and $\nabla \phi_i(v)$ have correlation $r > 0.5$ and $Var[B_t]\approx Var[\nabla \phi_i(v)]$, one can check that $Var[A_t] = Var[\nabla \phi_i(v) - B_k] =Var[\nabla \phi_i(v)] - 2r \sqrt{Var[\nabla \phi_i(v)] Var[B_k]} + Var[B_k] <Var[\nabla \phi_i(v)] $ (so the variance is reduced). 
Step 9 in Algorithm \ref{algo:PDSMD} uses this variance reduction technique by taking $B_k = \nabla \phi_{i}(\widetilde{{v}}^{s})$.
\\
The \textbf{acceleration} in Algorithm \ref{algo:PDSMD} are steps 7, 10, 11, namely the Katyusha acceleration in \cite{allen2017katyusha}. 
We summarize this technique and compare it with a classical method in \cite{allen2014linear} that uses Nesterov's momentum. 
To simplify explanation, consider the special case $\|\cdot\|_H = \|\cdot\|_2, w(\cdot) = \frac{1}{2}\|\cdot\|_2^2$, steps 7, 10, 11 of Algorithm \ref{algo:PDSMD} are:\\
\hspace*{40pt} $v_{k+1} = \tau_1 z_k + \tau_2 \tilde{v} + (1-\tau_1 - \tau_2)y_k; y_{k+1} = v_{k+1} - \frac{1}{3L} \Tilde{\nabla}_{k+1} ; z_{k+1} = z_k - \alpha \Tilde{\nabla}_{k+1},$\\
where $\mathbbm{E}\Tilde{\nabla}_{k+1} = \nabla \phi(v_{k+1})$. 
On the other hand, the method in \cite{allen2014linear} updates as \\
\hspace*{50pt}$v_{k+1} = \tau_1 z_k + (1-\tau_1)y_k; y_{k+1} = v_{k+1} - \frac{1}{L} \nabla \phi(v_{k+1}) ; z_{k+1} = z_k - \alpha \nabla \phi(v_{k+1}).$ 
\\
The two updating schemes both have a ``gradient descent'' step in $y_{k+1}$ and ``momentum'' term $z_{k+1}$ that accumulates the gradient history; the difference is in $v_{k+1}$: the classical method takes a weighted average of $z_k$ and $y_k$ (that is, Nesterov's momentum), while Katyusha acceleration has one more term $\Tilde{v}$ (which is called Katyusha momentum \cite{allen2017katyusha}). 
Such Katyusha momentum serves as a ``magnet'' to retract the estimation to $\Tilde{v}$, which is the average of past $l$ estimates. 
Since our algorithm is a stochastic algorithm, such a ``magnet'' helps the algorithm to stabilize. 
Thus, the Katyusha acceleration works well. 
\end{remark}

We prove the convergence rate of the PDASMD algorithm as follows:
\begin{theorem}\label{thm:convergence}
Under Assumption \ref{ass:dual}, we apply Algorithm \ref{algo:PDSMD} to solve problem \eqref{eq:obj}. 
Choose a mirror function $w(\cdot)$ that is $1$-strongly convex and $\gamma$-smooth w.r.t. $\|\cdot\|_{H}$ norm. Denote the primal and dual optimal solution as $\bm{x}^*$ and $\bm{\lambda}^*$, respectively.
Assume that $\|\lambda^*\|_{H}\leq R$. 
We have the convergence of the algorithm as follows:
\begin{align}
    &\| \mathbbm{E}[ \bm{b} - A\bm{x}^{S-1}]\|_{H,*}\leq\frac{2}{S^2 l}\left[ l\bar{L} R + 18\bar{L}R\gamma\right],\\
    & f( \mathbbm{E}(\bm{x}^{S-1}))-f(\bm{x}^*) \leq \frac{2}{S^2 l}\left[ l\bar{L} R^2 + 18\bar{L}R^2\gamma\right] .
\end{align}
\end{theorem}
The proof of the theorem is deferred to Appendix \ref{app:A}.

\subsection{Applying to Optimal Transport}\label{sec:2.3}
This section gives the detailed procedure of applying PDASMD to get an approximation solution to the OT. 
Especially we consider two cases: 
in the first case, we use $\|\cdot\|_H = \|\cdot\|_2$ and PDASMD reduce to PDASGD; 
in the second case, we use $\|\cdot\|_{H} = \|\cdot\|_{\infty}$ and prove an improved computational complexity over the first case.
Our algorithm achieves the best possible rate in the current literature for the latter case. 
Our algorithm improves the rate of the first-order algorithms for solving entropic OT. 

We apply the PDASMD algorithm to solve the entropic OT \eqref{eq:OT_pen} as follows.
Since problem \eqref{eq:OT_pen} a special case of problem \eqref{eq:obj}, we plug $A, \bm{b}, f(\cdot)$ into the general dual formula \eqref{eq:dual} to get the dual problem of problem \eqref{eq:OT_pen}.
With a little abuse of notation, we split the dual variables as $(\bm{\tau}^{T},\bm{\lambda}^{T})^{T}$ for $\bm{\tau},\bm{\lambda}\in \mathbb{R}^n$. 
The dual problem of problem \eqref{eq:OT_pen} is: 
\begin{equation}\label{eq:OT_dual}
    \phi(\bm{\tau},\bm{\lambda}) = \eta \langle\bm{1}_{n^2}, \bm{x}(\bm{\tau},\bm{\lambda})\rangle - \langle \bm{p}', \bm{\tau}\rangle - \langle \bm{q}', \bm{\lambda}\rangle,
\end{equation}
where the relationship between primal-dual variables is 
\begin{equation}\label{eq:dualvar}
\bm{x}(\bm{\tau},\bm{\lambda}) = \exp\left(\frac{A^T (\bm{\tau}^T,\bm{\lambda}^T)^T - \VEC(C) - \eta \mathbf{1}_{n^2}}{\eta}\right).
\end{equation}
Moreover, to get a dual with the finite-sum structure, we follow \cite{genevay2016stochastic} to transfer the dual objective to semi-dual by fixing $\bm{\lambda}$ and solving the first order condition w.r.t. $\bm{\tau}$ in objective \eqref{eq:OT_dual}. 
This gives us the relationship between the dual variables:
\[\tau_i(\bm{\lambda}) = \eta\log p'_i - \eta \log \left(\sum_{j=1}^n \exp((\lambda_j - C_{i,j} - \eta)/\eta)\right). \]

Plugging the relationship above into the dual objective \eqref{eq:OT_dual} gives us the semi-dual objective. 
With a little abuse of notation, we denote the semi-dual objective function as $\phi(\bm{\lambda})$, which is:
\begin{align}
\phi(\bm{\lambda}) &= - \langle \bm{q}', \bm{\lambda}\rangle - \eta\sum_{i=1}^n p'_i\log p'_i \nonumber\\
&\quad + \eta\sum_{i=1}^n \log \left(\sum_{j=1}^n \exp((\lambda_j - C_{i,j} - \eta)/\eta)\right)+ \eta  \nonumber\\
&= \frac{1}{n}\sum_{i=1}^n np_i'\Bigg[- \langle \bm{q}', \bm{\lambda}\rangle - \eta\log p'_i\nonumber\\
& \quad+ \eta\log \left(\sum_{j=1}^n \exp((\lambda_j - C_{i,j} - \eta)/\eta)\right) + \eta\Bigg]\nonumber\\
&:= \frac{1}{n}\sum_{i=1}^n \phi_i(\bm{\lambda}).\label{eq:semi-dual}
\end{align}
It is easy to check that each $\phi_i(\bm{\lambda})$ is convex. 
To apply our algorithm, we further check the smoothness of $\phi_i(\bm{\lambda})$ in the following lemma:
\begin{lemma}\label{lem:smooth}
$\phi_i(\cdot)$ in the semi-dual objective \eqref{eq:semi-dual} is $\frac{np_i'}{\eta}$ smooth w.r.t. $\|\cdot\|_2$ norm, and is $\frac{5np_i'}{\eta}$ smooth w.r.t. $\|\cdot\|_{\infty}$ norm.
\end{lemma}
Lemma \ref{lem:smooth} is proved in Appendix \ref{app:B}. 
By Lemma \ref{lem:smooth}, we can calculate the parameter in PDASMD Algorithm \ref{algo:PDSMD} as $\bar{L} = 1/\eta$ for $\|\cdot\|_H = \|\cdot\|_2$, and $\bar{L} = 5/\eta$ for $\|\cdot\|_H = \|\cdot\|_{\infty}$. 
For these two cases, we can apply Algorithm \ref{algo:PDSMD} to approximate problem \eqref{eq:OT_pen}. 
We further round the approximating solution of problem \eqref{eq:OT_pen} to the feasible region of problem \eqref{eq:OT_nopen}. 
This way, we get an $\epsilon-$solution to problem \eqref{eq:OT_nopen}. 
The full procedure is deferred to Appendix \ref{app:C} due to the page limit.
We state the computational complexity of the full procedure in the following theorem:

\begin{theorem}\label{thm:computation_ot}
 Set $l = \Theta(n)$ in the PDASMD algorithm, the overall number of arithmetic operations for finding a solution $\widehat{X}$ such that 
$\mathbbm{E}\langle C,\widehat{X}\rangle \leq\langle C,X^*\rangle +\epsilon$
is 
\begin{itemize}[leftmargin=*,noitemsep]
    \item $\mathcal{\widetilde{O}}\left(\frac{n^{2.5}\|C\|_{\infty}(1+\sqrt{\gamma/n})}{\epsilon}\right)$ for $\|\cdot\|_H = \|\cdot\|_2$;
    \item $\mathcal{\widetilde{O}}\left(\frac{n^{2}\|C\|_{\infty} ( 1 +\sqrt{\gamma/n})}{\epsilon}\right)$ for $\|\cdot\|_H = \|\cdot\|_\infty$.
\end{itemize}
\end{theorem}
The proof of Theorem \ref{thm:computation_ot} is in Appendix \ref{app:C}.
\begin{remark}\label{remark:03}
The complexities still depend on $\gamma$, the smoothness of $w(\cdot)$ w.r.t. $\|\cdot\|_H$. 
For example, when taking $w(\cdot) = \frac{1}{2}\|\cdot\|_2^2$, we have $\gamma = 1$ for $\|\cdot\|_H =\|\cdot\|_2$, and $\gamma = n$ for $\|\cdot\|_H =\|\cdot\|_\infty$. 
The corresponding computational complexity is then  $\mathcal{\widetilde{O}}\left(\frac{n^{2.5}\|C\|_{\infty}}{\epsilon}\right)$ and $\mathcal{\widetilde{O}}\left(\frac{n^{2}\|C\|_{\infty}}{\epsilon}\right)$. 
Now for $\|\cdot\|_H = \|\cdot\|_\infty$, as long as we choose a proper $w(\cdot)$ such that $\gamma = \mathcal{O}(n)$, the rate $\mathcal{\widetilde{O}}\left(\frac{n^{2}\|C\|_{\infty}}{\epsilon}\right)$ is achieved.
One may further improve the rate by a constant by improving the dependency of $\gamma$ on $n$. 
Such improvement is an open question in optimization; though we make no effort to do it in this paper, we still note this opportunity. 
\end{remark}

\begin{remark}
If we choose $w(\cdot) = \frac{1}{2}\|\cdot\|_2^2$, we have closed-form solutions for each step of PDASMD. 
\begin{itemize}[leftmargin=*,noitemsep]
    \item For both settings, step 10 of PDASMD algorithm becomes $\bm{z}_{k+1} = \bm{z}_k - \alpha_s \widetilde{\bm{\nabla}}_{k+1}$;
    \item For $\|\cdot\|_H = \|\cdot\|_2$, step 11 of PDASMD is $\bm{y}_{k+1} = \bm{v}_{k+1} - \frac{1}{9\bar{L}}\widetilde{\bm{\nabla}}_{k+1}$;
    \item For $\|\cdot\|_H = \|\cdot\|_\infty$, step 11 of PDASMD becomes $\bm{y}_{k+1} = \bm{v}_{k+1} - \frac{\|\widetilde{\bm{\nabla}}_{k+1}\|_1}{9\bar{L}} sign(\widetilde{\bm{\nabla}}_{k+1})$.
\end{itemize}
It is clear that in both settings, each step of PDASMD is defined by simple arithmetic operations and thus is easy to implement. 
There is no gap between our theory and practice. 
\end{remark}

\subsection{Computational Complexity of the Stochastic Sinkhorn}
In this section, we prove that the computational complexity of the Stochastic Sinkhorn for finding an $\epsilon$-solution to OT is $\widetilde{\mathcal{O}}(\frac{n^2}{\epsilon^2})$, which is improved over the known rate of $\widetilde{\mathcal{O}}(\frac{n^{2}}{\epsilon^3})$ \cite{abid2018stochastic} and matches the state-of-the-art rate of Sinkhorn and Greenkhorn \cite{dvurechensky2018computational,lin2019efficient}. 
Moreover, our PDASMD algorithm beats the provable rate of Stochastic Sinkhorn.
This illustrates the advantage of our PDASMD algorithm.

The Stochastic Sinkhorn algorithm is proposed by \cite{abid2018stochastic}.
One can check Appendix \ref{app:E} for a full algorithm description. 
We show the computational complexity of the Stochastic Sinkhorn as follows:
\begin{theorem}\label{thm:stoc_sinkhorn}
Stochastic Sinkhorn finds a solution $\widehat{X}$ such that 
$\mathbbm{E}\langle C,\widehat{X}\rangle \leq\langle C,X^*\rangle +\epsilon$
in $${\mathcal{O}}\left(\frac{n^2 \|C\|^2_{\infty}\log n}{\epsilon^2}\right)$$ arithmetic operations.
\end{theorem}
The proof of Theorem \ref{thm:stoc_sinkhorn} is in Appendix \ref{app:E}.

\begin{figure*}
%\vskip -0.2in
\begin{center}
    \subfigure[]{\includegraphics[width=2.35in]{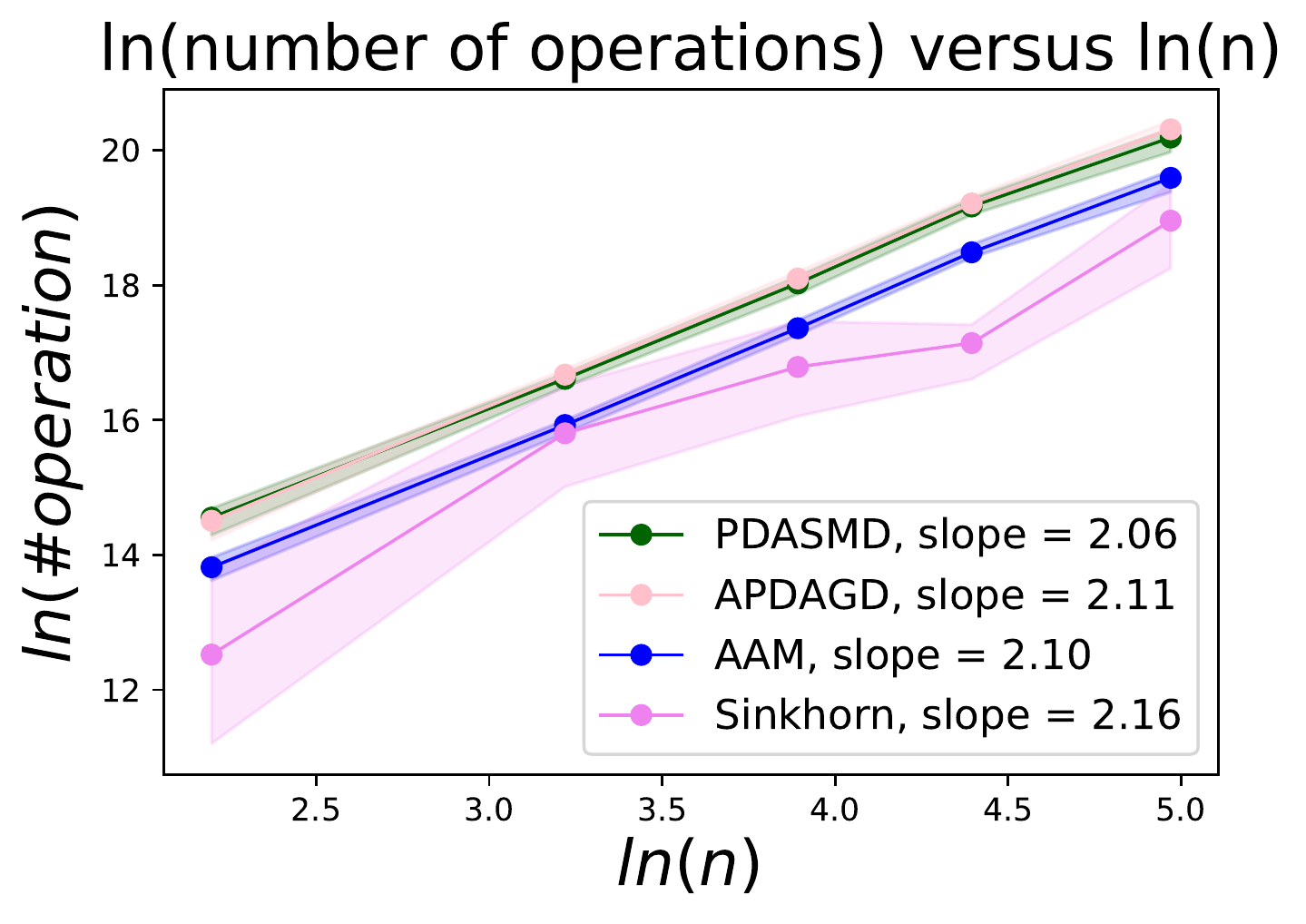}\label{fig:1a}} 
    \subfigure[]{\includegraphics[width=2.32in]{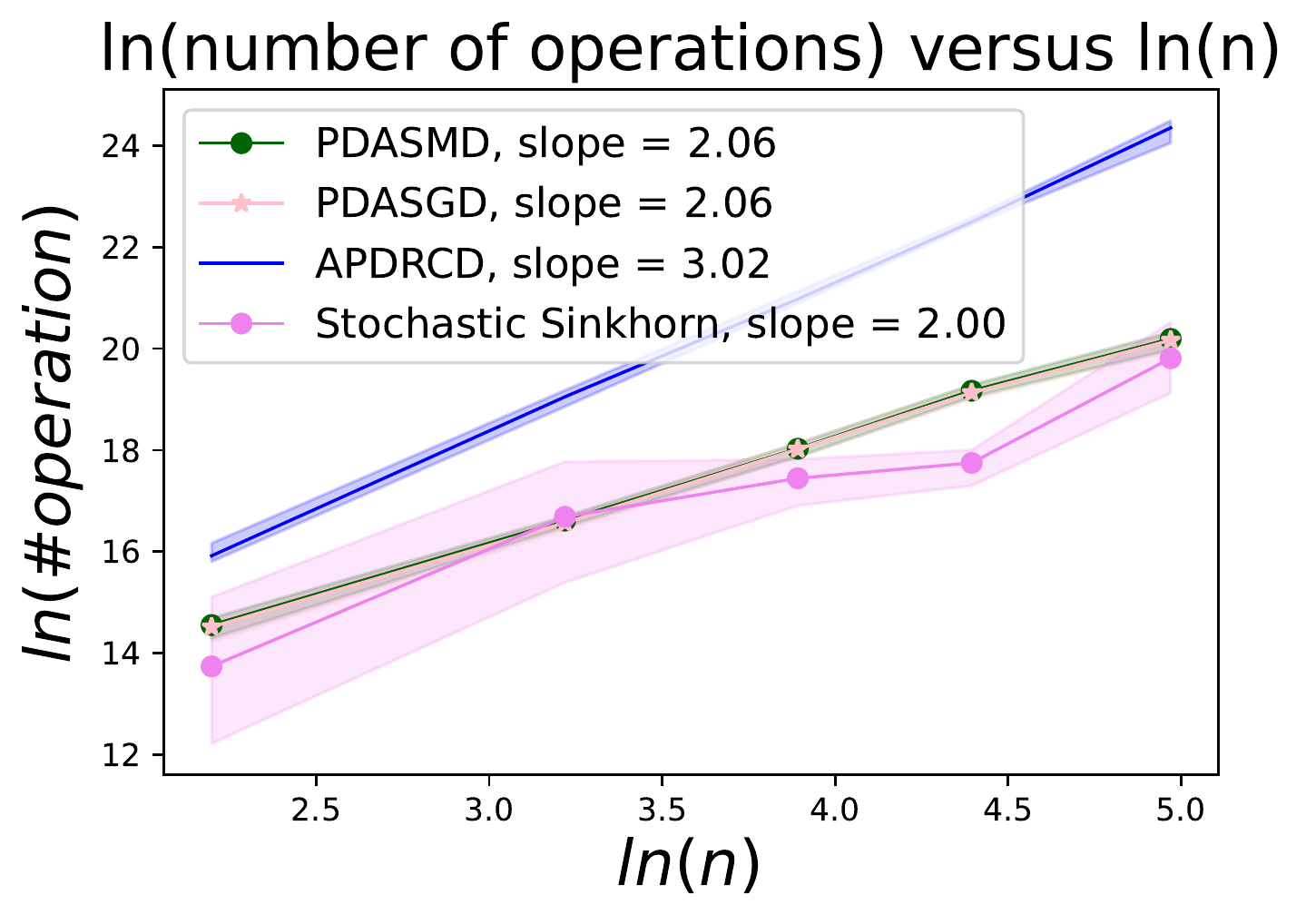}\label{fig:1b}} 
    \subfigure[]{\includegraphics[width=2.6in]{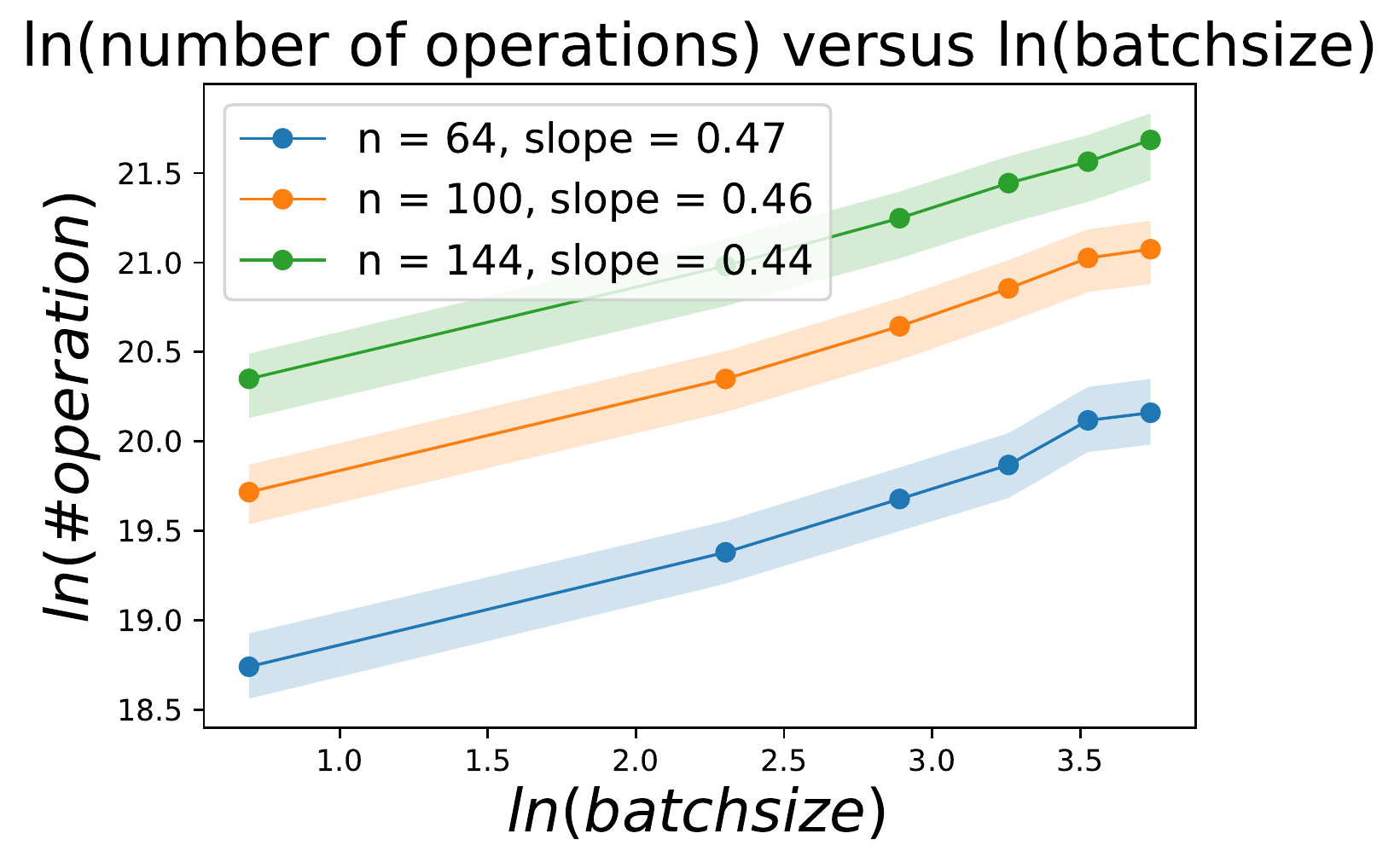}\label{fig:1c}} 
    \subfigure[]{\includegraphics[width=2.25in]{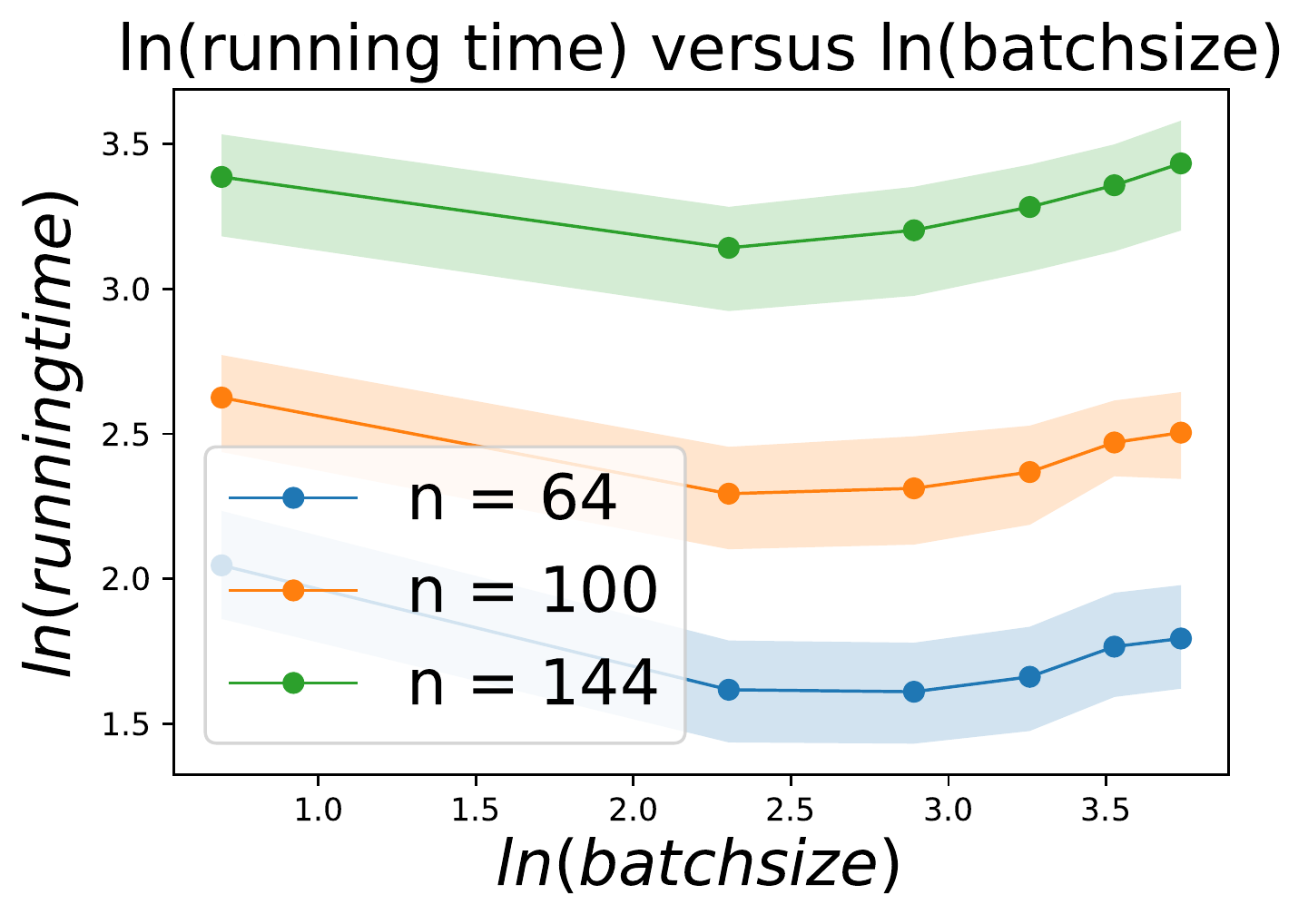}\label{fig:1d}} 
    \subfigure[]{\includegraphics[width=2.35in]{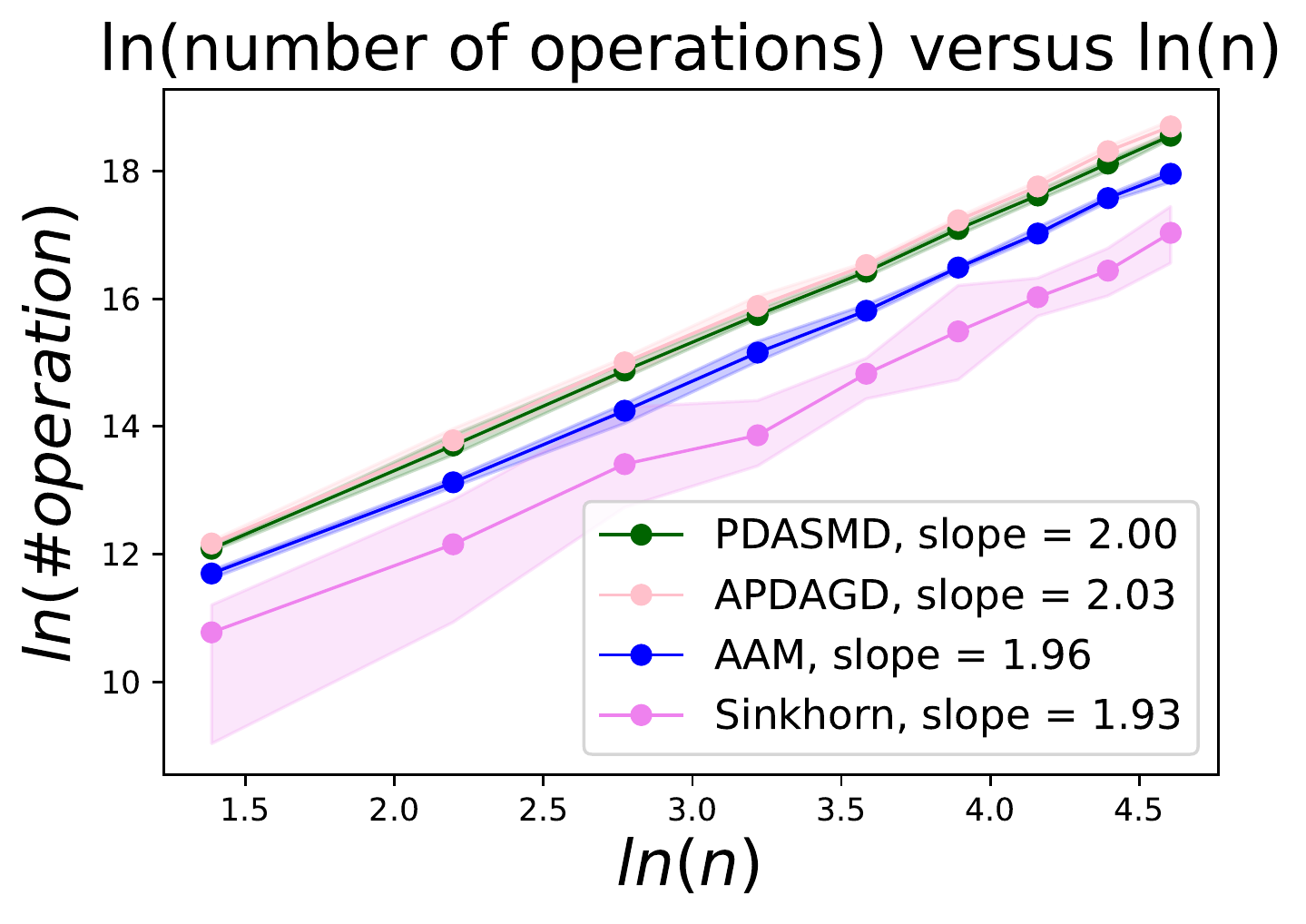}\label{fig:1e}} 
    \subfigure[]{\includegraphics[width=2.32in]{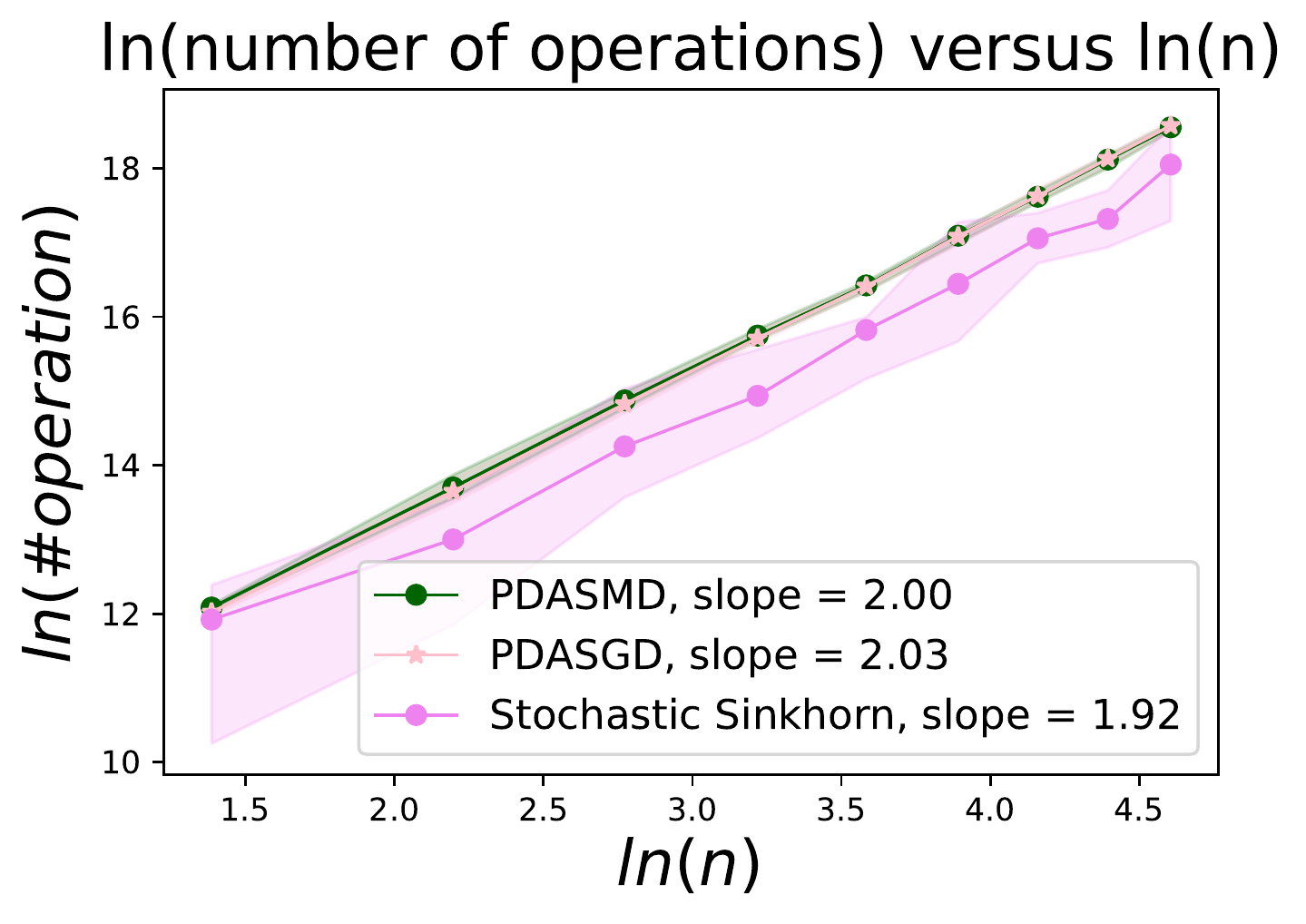}\label{fig:1f}} 
    \subfigure[]{\includegraphics[width=2.6in]{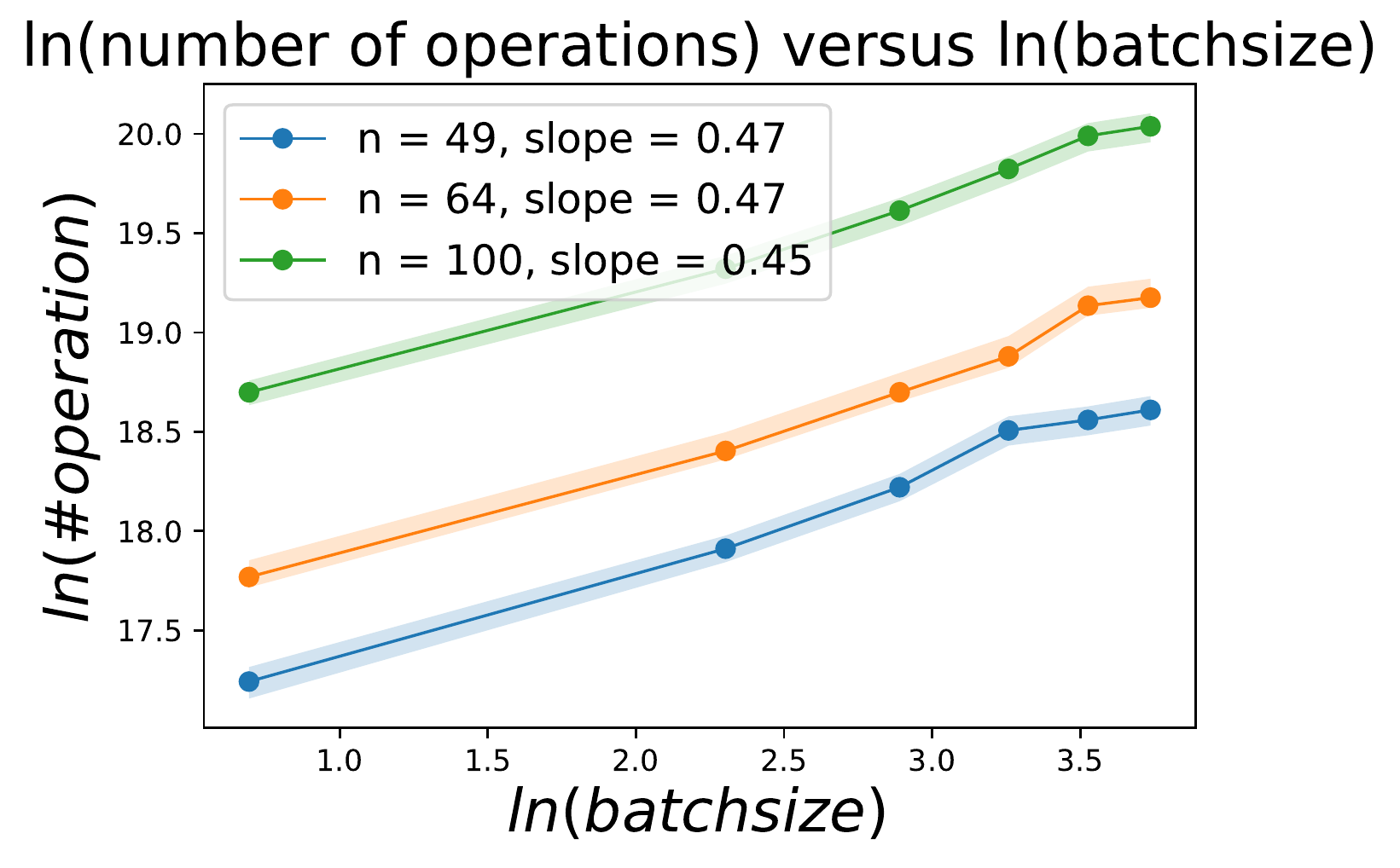}\label{fig:1g}} 
    \subfigure[]{\includegraphics[width=2.25in]{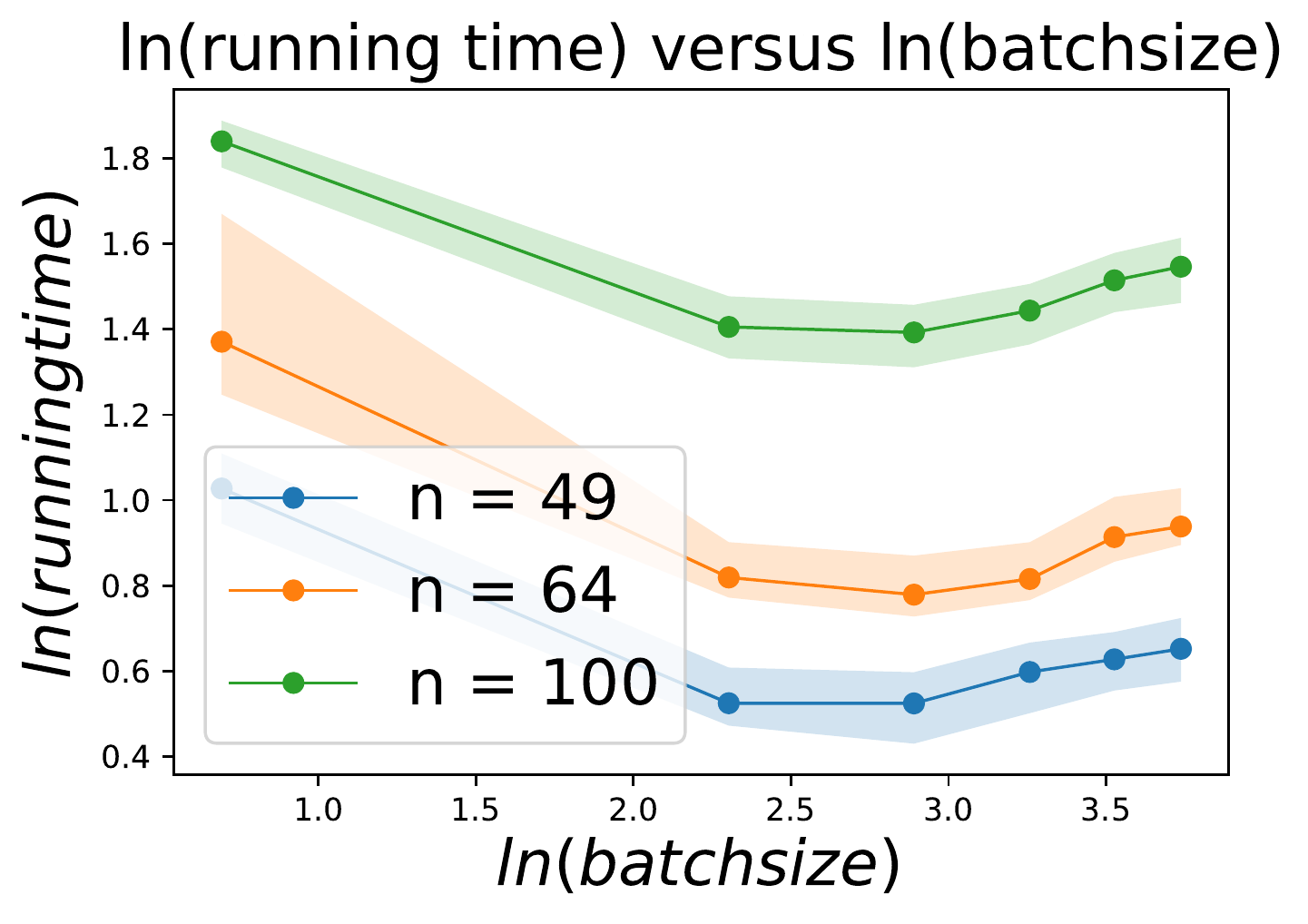}\label{fig:1h}}
\caption{
Computational complexity comparison of different algorithms for finding an $\epsilon$-solution of OT. 
The logarithmic of the total number of numerical operations to achieve a given $\epsilon$ approximation error is plotted against either the logarithmic transform of the sample size $n$ in the PDASMD algorithm (rows 1 and 3) or the batch size in the PDASMD-B algorithm (rows 2 and 4).  
The first two rows use synthetic data, and the last two are for the MNIST data.
The relevant discussion can be seen in Section \ref{sec:5} Numerical Studies. 
The error bars in all the plots come from repeating the experiment using 5 pairs of randomly generated/chosen marginals.}
\label{fig:1}
\end{center}    
%\vskip -18pt    
\end{figure*}

\section{PDASMD with Batch Implementation (PDASMD-B)}\label{sec:3}
In this section, we propose a batch version of PDASMD, namely the PDASMD-B algorithm. 
The batch implementation of the stochastic step in PDASMD-B allows parallel computing. 
This further improves the computational power of our algorithm.

\begin{algorithm}[h]
   \caption{Batch PDASMD (PDVRASMD-B)}
   \begin{algorithmic}[1]
   \label{algo:PDSMD-B}
   \STATE Initialize $l$ the number of inner iterations, $B$ the batch size; set $\tau_2\leftarrow \frac{1}{2 B}$, $C_0=D_0= 0$, $\bm{y}_0=\bm{z}_0=\widetilde{\bm{v}}_0=\bm{v}_0 = \bm{0}$;
    choose a mirror function $w(\cdot)$ that is 1-strongly convex and $\gamma$-smooth w.r.t. $\|\cdot\|_{H}$, and denote by $V_{\bm{x}}(\bm{y})$ the Bregman divergence generated by $w(\cdot)$; 
   take $\bar{L} = (\sum_{i=1}^m L_i)/m$, where $L_i$ is the smoothness (w.r.t. $\|\cdot\|_H$) for each component $\phi_i$ of the dual function $\phi(\cdot)$ in Assumption \ref{ass:dual}. 
   \FOR{s = 0,\ldots,S-1}
   \STATE $\tau_{1,s}\leftarrow {2}/(s+4)$; $\alpha_s\leftarrow {1}/(9 \tau_{1,s}\bar{L})$;
   \STATE $\bm{\mu}^{s}\leftarrow \nabla \phi(\widetilde{\bm{v}}^s)$.
   \FOR{$j=0$ to $l-1$}
   \STATE $k\leftarrow (sl)+j$;
   \STATE $\bm{v}_{k+1}\leftarrow\tau_{1,s}\bm{z}_k+\tau_2\widetilde{\bm{v}}^{s}+(1-\tau_{1,s}-\tau_2)\bm{y}_k$;   
   \STATE Pick $B$ samples independently from $\{1, 2, \ldots, m\}$ with replacement, where sample $i$ is picked with probability $p_i = L_i/m\bar{L}$; denote the sampled index set as $I$;
   \STATE $\widetilde{\bm{\nabla}}_{k+1}\leftarrow \bm{\mu}^{s}+\frac{1}{B}\sum_{i\in I}\frac{1}{mp_i}(\nabla \phi_{i}(\bm{v}_{k+1})-\nabla \phi_{i}(\widetilde{\bm{v}}^{s}))$;
   \STATE $\bm{z}_{k+1}=\arg\min_{\bm{z}}\{\frac{1}{\alpha_s}V_{ \bm{z}_{k}}(\bm{z})+\langle\widetilde{\bm{\nabla}}_{k+1},\bm{z}\rangle\}$;
   \STATE $\bm{y}_{k+1}=\arg\min_{\bm{y}}\{ \frac{9\bar{L}}{2}\|\bm{y}- \bm{v}_{k+1}\|_{H}^2+\langle\widetilde{\bm{\nabla}}_{k+1},\bm{y}\rangle \}$.
   \ENDFOR
   \STATE $\widetilde{\bm{v}}^{s+1}\leftarrow \frac{1}{l}\sum\limits_{j=1}^{l}\bm{y}_{sl+j}$;
   \STATE $C_{s} \leftarrow C_s+\frac{1}{\tau_{1,s}}$;
    \STATE Pick $\widetilde{\bm{y}}_{s}$ uniform randomly from $\{\bm{y}_{sl+j}\}_{j=1}^l$, update $D_s \leftarrow D_s+\frac{1}{\tau_{1,s}} \bm{x}(\widetilde{\bm{y}}_{s})$, where $\bm{x}(\cdot)$ is given by equation \eqref{eq:pd_relation};
   \STATE $\bm{x}^s={D_s}/{C_s}$.
   \ENDFOR
   \STATE \textbf{Output:} $\widetilde{\bm{x}} = \bm{x}^{S-1}$.
   \end{algorithmic}
\end{algorithm}

We give PDASMD-B in Algorithm \ref{algo:PDSMD-B} and briefly explain it. 
As compared to the non-batch version PDASMD in Algorithm \ref{algo:PDSMD}, Step 8 of PDASMD-B now samples a small batch of samples and calculates $\widetilde{\bm{\nabla}}_{k+1}$ based on the gradient of this small batch. 
Other hyper-parameters in the algorithm are changed accordingly to ensure convergence.

We apply PDASMD-B to solve OT. 
The main steps are the same as those in Subsection \ref{sec:2.3}; thus, we omit the details. 
To compute the computational complexity for giving an $\epsilon$-solution to OT, one needs the convergence result of PDASMD-B, which we include in Appendix \ref{app:D}. 
And the computational complexity for solving OT is stated in the following corollary. 

\begin{corollary}\label{cor:computation_ot_b}
Run PDASMD-B with batch size $B$, $\|\cdot\|_H = \|\cdot\|_\infty$ and inner loop size $l = n/B$ (assume w.l.o.g. that $l$ is an integer), the overall number of arithmetic operations to find a solution $\widehat{X}$ such that 
$\mathbbm{E}\langle C,\widehat{X}\rangle \leq\langle C,X^*\rangle +\epsilon$ is
 \[\mathcal{\widetilde{O}}\left(\frac{n^{2}\|C\|_{\infty}  \sqrt{1/B + B\gamma/n}}{\epsilon}\right).\] 
\end{corollary}

\begin{remark}
Corollary \ref{cor:computation_ot_b} shows the speed-up of PDASMD-B from parallel computing. 
We analyzed the speed-up for two cases of $\gamma$ as follows. 
The first case is similar to the one in Remark \ref{remark:03}: taking $w(\cdot) = \frac{1}{2}\|\cdot\|_2^2$, then we have $\gamma = n$. 
This gives us the total computation of
$\mathcal{\widetilde{O}}\left(\frac{n^{2}\|C\|_{\infty}\sqrt{B}}{\epsilon}\right) $, which is $\sqrt{B}$ times that of non-batch version.
There are $B$ batches of parallel computation, so if we ignore the communication time, our batch algorithm enjoys a sublinear speed-up of $\mathcal{O}(\sqrt{B})$.
The second case assumes one can further improve the rate $\gamma \sim \mathcal{O}(n)$ to $\gamma \sim \mathcal{O}(\sqrt{n})$. 
Then for $B \leq \sqrt{n}$, the number of total computations does not increase with $B$, which indicates a linear speed-up of $\mathcal{O}(B)$ using parallel computing.
Though such an improvement in $\gamma$ is still an open question in optimization, this implies a potentially huge advantage of the batch algorithm.
\end{remark}

\section{Numerical Studies}\label{sec:5}
In this section, we discuss the result of our numerical studies. 
The goals of our experiment are to check our theoretical computational complexity of the PDASMD algorithm w.r.t. the marginal size $n$ in Theorem \ref{thm:computation_ot}, and to check the theoretical computational complexity of the PDASMD-B algorithm w.r.t. the batch size $B$ in Corollary \ref{cor:computation_ot_b}. 
We use both synthetic and real grey-scale images \footnote{The MNIST dataset \cite{lecun1998mnist}.} as the marginal distribution for our experiment. 
Due to the page limit, our data description and algorithm implementation are deferred to Appendix \ref{app:G}. 
We have more applications of our algorithm, including domain adaptation and color transfer, in Appendix \ref{app:H}. 

Our experiment results are given in Figure \ref{fig:1}. 
We now explain the plots and summarize the results from the plots as follows. 

Figures \ref{fig:1a}, \ref{fig:1b}, \ref{fig:1e} and \ref{fig:1f} check the computational complexity of PDASMD on the marginal size $n$. 
In our experiment, we run PDASMD with $w(\cdot) = \frac{1}{2}\|\cdot\|_2^2$ and $\|\cdot\|_H =\|\cdot\|_\infty$. 
By Theorem \ref{thm:computation_ot}, for this case, when fixing the accuracy level $\epsilon$, we should have the computational complexity $\sim \mathcal{O}(n^2)$. 
That is, fixing a $\epsilon$ and plotting the logarithm of computation count versus the logarithm of $n$, we expect to see a line with slope $2$. 
Figures \ref{fig:1a}, \ref{fig:1b} (using synthetic data as marginals) and Figures \ref{fig:1e} and \ref{fig:1f} (using real data as marginals) have the lines corresponding to the PDASMD algorithm have slopes that are close to $2$, which supports our theoretical rate. 

In Figures \ref{fig:1a}, \ref{fig:1b}, \ref{fig:1e} and \ref{fig:1f} we also include lines that correspond to other state-of-the-art algorithms. 
The goal is to compare the practical performance of the PDASMD algorithm with deterministic algorithms (Figures \ref{fig:1a} and \ref{fig:1e}) and other stochastic algorithms (Figure \ref{fig:1b} and \ref{fig:1f}). 
We conclude from the plots that the total computation numbers of the AAM, Sinkhorn, and Stochastic Sinkhorn are less than that of the PDASMD, which illustrates the practical advantage of those algorithms. 
However, such an observation does not disqualify our PDASMD algorithm since we still have a provable complexity that is better than those algorithms. 
Inspired by such an observation, one may further improve the PDASMD in practice. 
One possible way is to combine the PDASMD algorithm with the Sinkhorn to take advantage of the better theoretical rate of PDASMD and the good empirical performance of the Sinkhorn. 

Figures \ref{fig:1c} and \ref{fig:1g} check the computational complexity of PDASMD-B on the batch size $B$. 
We fix the accuracy level $\epsilon$ and run PDASMD-B with $w(\cdot) = \frac{1}{2}\|\cdot\|_2^2$. 
By Corollary \ref{cor:computation_ot_b}, for a given marginal size $n$, we have the number of total computation $\sim \mathcal{O}(\sqrt{B})$. 
Thus, when plotting the logarithm of computation count versus the logarithm of $B$, we should get a line with a slope $0.5$. 
In Figures \ref{fig:1c} (using synthetic data as marginals) and \ref{fig:1g} (using real data as marginals), we see that for different marginal sizes $n$, the slopes are all close to $.5$. 
Such an observation matches our theory. 

With such computational complexity of PDASMD-B on the batch size $B$, if we can fully parallelize, the running time of PDASMD-B should be $\sim \mathcal{O}(B^{-0.5})$.  
To check this, we plot the logarithm of running time versus the logarithm of $B$ in Figures \ref{fig:1d} and \ref{fig:1h}. 
The lines fail to have slope $-0.5$. 
This is not surprising to see in practice because of the commutation time and limit in the computational resource. 
But from the plots, we can still benefit from the batch algorithm: when the batch size is not too large ($<= \exp(2.5)$), the running time decreases as the batch size increases. 
This illustrates the usefulness of the batch version algorithm in practice. 

To summarize, our computational complexity of PDASMD on $n$ and PDASMD-B on $B$ are supported by numerical studies. 
%The experiment points out the future directions that one can further improve our algorithm in practice 
%In the mean time, our experiment points out possible future study of: 1. improve our algo in practice or other algos in theory, combining them might be one way to go; 2. improve the batch version of our algo: in practice do a full parallleing to take advantage of our sublinear speed-up, or in theory propose batch algo with linear speed-up. 

\section{Discussion and Future Study}\label{sec:6}
This paper proposes a new first-order algorithm for solving entropic OT. 
We call our algorithm the PDASMD algorithm. 
We prove that our algorithm finds an $\epsilon$-solution to OT using $\widetilde{\mathcal{O}}(n^2/\epsilon)$ arithmetic operations.
Such a rate improves the previously state-of-the-art rate of $\widetilde{\mathcal{O}}(n^{2.5}/\epsilon)$ among the first-order algorithms applied to entropic OT. 
We perform numerical studies, and the results match our theory.

We discuss some future directions for improving the computational efficiency of OT.

One direction is to revisit other first-order algorithms that are proved to have $\widetilde{\mathcal{O}}(n^{2.5}/\epsilon)$ computational complexity, and see if they can be improved to $\widetilde{\mathcal{O}}(n^{2}/\epsilon)$. 
Some algorithms show the $\widetilde{\mathcal{O}}(n^2/\epsilon)$ rate in practice, but there is no proof for such a rate. 
The techniques in our paper may inspire proper modifications to those algorithms to get a better provable rate. 
In this way, one may further prove a computational complexity better than that of the PDASMD algorithm by a constant. 

Another direction is to combine our algorithm with iterative projection-based algorithms such as the Sinkhorn. 
This direction is motivated by the Accelerated Sinkhorn algorithm in \cite{lin2022efficiency}, which updates the dual variables of entropic OT by Nesterov’s estimate sequence (for acceleration) and two Sinkhorn steps. 
Now our PDASMD algorithm also uses an acceleration technique (Katyusha momentum), it would be interesting to analyze a stochastic Accelerated Sinkhorn by replacing its Nesterov’s estimate sequence with the Katyusha momentum.

The third direction is to improve the batch version of our PDASMD algorithm. 
Our batch-version algorithm has a sub-linear speed-up when fully parallelized and ignores the communication time. 
In such a setting, one may expect an optimally designed batch algorithm to speed up linearly. 
That is, the total number of computations does not scale up with the batch size, and the computing time is $1/B$ that of the non-batch version when the batch size is $B$. 
If one can improve our batch version algorithm to achieve a linear speed-up, the computational advantage will be huge. 

Besides computing for OT, the broader applications of our PDASMD algorithm are also interesting. 
Our PDASMD algorithm can be applied to a linear constrained strongly convex problem as long as its dual is of a finite-sum form. 
This motivates one to apply our algorithm to solve other problems such as the unbalanced OT \cite{pham2020unbalanced} and the Wasserstein barycenter \cite{cuturi2014fast} for better computational complexity.

\bibliography{mirror_ot}

%\end{document}

%%%%%%%%%%%%%%%%%%%%%%%%%%%%%%%%%%%%%%%%%%%%%%%%%%%%%%%%%%%%

\newpage
\appendix

\section{Proximal Mirror Descent}\label{app:PMD}
In this section, we review the technique of stochastic proximal mirror descent. 

Let us start with the objective function: 
\begin{equation}\label{eq:ERM}
    \min_{\bm{x}} F(\bm{x}) := \frac{1}{n}\sum_{i=1}^n f_i(\bm{x}).
\end{equation}
A popular way to minimize problem \eqref{eq:ERM} is the Stochastic Gradient Descent (SGD).
At time $t$, the SGD algorithm randomly samples $i_t$ from $\{1,\ldots,n\}$ and updates as: 
\begin{equation}\label{eq:SGD}
\bm{x}_{t+1} = \bm{x}_{t} -b_{t} \nabla f_{i_t}(\bm{x}_t),
\end{equation}
where $b_{t}$ is the step size. Note that formula \eqref{eq:SGD} is essentially the solution to the following $\ell_2$ penalized problem:
\begin{equation}\label{eq:SGD_rewrite}
    \bm{x}_{t+1} =\arg\min_{\bm{x}} \left\{ \langle \bm{x}, \nabla f_{i_t}(\bm{x}_t)\rangle +  \frac{1}{2b_t} \|\bm{x} - \bm{x}_t\|_2^2\right\}.
\end{equation}
The proximal/mirror descent is proposed by Nemirovski and Yudin \cite{nemirovskii1983problem}, where they generalize the SGD by replacing the $\|\cdot\|_2^2$ term in problem \eqref{eq:SGD_rewrite} by some proximity function. 
There are two popular choices of proximity functions, and they lead to stochastic proximal and mirror descent, respectively. In this paper, we use stochastic proximal mirror descent to represent both cases.

The choice of proximity function that leads to stochastic proximal gradient descent is the square of an arbitrary norm $\|\cdot\|_H$ (as compared to the $\|\cdot\|_2$ norm in problem \eqref{eq:SGD_rewrite}). 
This results in the update
\begin{equation}\label{eq:ProxD}
    \bm{x}_{t+1} =\arg\min_{\bm{x}} \left\{ \langle \bm{x}, \nabla f_{i_t}(\bm{x}_t)\rangle +  \frac{1}{2b_t} \|\bm{x} - \bm{x}_t\|_H^2\right\}.
\end{equation}

The choice of proximity function that gives stochastic mirror descent is the Bregman divergence. 
Recall that for a mirror map $w(\cdot)$, the Bregman divergence is $$V_{\bm{x}}(\bm{y}) := w(\bm{y}) - w(\bm{x}) - \langle \nabla w(\bm{x}), \bm{y} - \bm{x}\rangle.$$
The stochastic mirror descent then updates as:
\begin{equation}\label{eq:MD}
    \bm{x}_{t+1} =\arg\min_{\bm{x}} \left\{ \langle \bm{x}, \nabla f_{i_t}(\bm{x}_t)\rangle +  \frac{1}{2b_t} V_{\bm{x}_t}(\bm{x})\right\}.
\end{equation}
Note that the popular KL-divergence
$\mathcal{KL}(\bm{x}||\bm{x}') := \sum_i x_i\log({x_i}/{x'_i}) - \sum_i x_i + \sum_i x'_i $
is a special case of Bregman divergence by choosing $w$ to be the negative entropy $w(\bm{x}) = \sum_i x_i \log x_i$.

Recall that the objective is to solve the (entropic) OT, so we explain why the proximal mirror algorithm might be suitable for optimizing the entropic OT compared with the SGD. 

First, the objective function of entropic OT coincides with the proximal mirror descent formulation in that each step of proximal mirror descent minimizes an inner product term plus a divergence term other than the $\ell_2$ norm.
In this way, the proximal mirror descent may help to prove a faster convergence when solving OT. 

Second, it is pointed out that the popular Sinkhorn algorithm to solve entropic OT can be interpreted as a special case of the stochastic proximal mirror descent algorithm \cite{mishchenko2019}. 
We briefly summarize their statement as follows. 
The Sinkhorn algorithm iteratively updates the dual variables $\bm{u},\bm{v}$ of problem \eqref{eq:OT_pen} by:
\begin{equation}\label{eq:sinkhorn_eq1}
    \bm{u}^{k+1} = \bm{u}^k + log \bm{p}' - log(X(\bm{u}^k,\bm{v}^k) \bm{1}), \bm{v}^{k+1} = \bm{v}^{k} ,
\end{equation}
and
\begin{equation}\label{eq:sinkhorn_eq2}
    \bm{u}^{k+1} = \bm{u}^k, \bm{v}^{k+1} = \bm{v}^k + log \bm{q}' - log(X(\bm{u}^k,\bm{v}^k)^T \bm{1}),
\end{equation}
where the relationship between the primal-dual variables is
$$ X(\bm{u},\bm{v}) = diag(\exp(\bm{u})) \exp(-C/\eta) diag(\exp(\bm{v})). $$
Notice that the dual variables $\bm{u}, \bm{v}$ are equivalent to the dual variables $\bm{\lambda},\bm{\tau}$ we use in formulation \eqref{eq:dualvar} plus constants.

To interpret Sinkhorn as a Stochastic Mirror Descent, one considers the objective function:
\begin{align}\label{eq:reform}
    \min_{X \in \mathbb{R}^{n\times n}} f(X) &:= \frac{1}{2}(f_1(X) + f_2(X))\\
    f_1(X)&= \mathcal{KL}(X \bm{1}||\bm{p}'), f_2(X) = \mathcal{KL}(X^T \bm{1}||\bm{q}').
\end{align}
Now the objective function is a finite sum of two functions: $f_1(\cdot)$ and $f_2(\cdot)$, then we can run SMD on it. 
Suppose that the SMD is initialized at $X_0 = \exp(-C/\eta)$, choose the step size $\eta=1$ and mirror map $w(X) = \sum_{i,j} X_{i,j}(\log X_{i,j} - 1)$. 
When the first sample is used (i.e. the sub-gradient of $f_1$ is used), SMD updates as
\begin{align*}
    \nabla w(X^{k+1}) = \nabla w(X^{k}) - \nabla f_1(X^k).
\end{align*}
One can check that it is exactly equivalent to one step Sinkhorn update in $\bm{u}$ as step \eqref{eq:sinkhorn_eq1}. 
Similarly, SMD using $f_2$ is equivalent to one step Sinkhorn update in $v$ as step \eqref{eq:sinkhorn_eq2}.

From the above, the Sinkhorn is a special case of SMD, which suggests that mirror-based algorithms may be proper for solving the entropic OT. 
Given the success of the Sinkhorn algorithm, it would be interesting to discover more general stochastic proximal mirror descent algorithms and study their performance for solving OT.

%In fact, we can compare the formulations \eqref{eq:OT_nopen} and \eqref{eq:reform} to generate another interesting finding: the original constraint in \eqref{eq:OT_nopen} becomes the new objective \eqref{eq:reform}, and the original objective is encoded into the initialization of SMD algorithm. 
%This can be interpreted as implicit regularization of SMD for the penalized OT problem.

\section{Proof for Theorem \ref{thm:convergence}}\label{app:A}

To prove Theorem \ref{thm:convergence}, the following lemmas are established:

\begin{lemma}[Coupling step 1] \label{lem:01} Consider one inner loop of Algorithm \ref{algo:PDSMD}, where the randomness only comes from the choice of $i$. It satisfies that for $\forall u$:
\begin{align*}
    &\alpha_s \langle \nabla \phi(v_{k+1}), z_k - u\rangle\\
\leq &\frac{\alpha_s}{\tau_{1,s}} \left\{\phi(v_{k+1}) - \mathbbm{E}[\phi(y_{k+1})] + \tau_2 \phi(\widetilde{v}^s) - \tau_2 \phi(v_{k+1}) - \tau_2 \langle\nabla\phi(v_{k+1}),\widetilde{v}^s - v_{k-1}\rangle\right\}\\
& \quad + V_{z_k}(u) - \mathbbm{E}[V_{z_{k+1}}(u)].
\end{align*}
\end{lemma}
\begin{proof}
Note that the inner loop of Algorithm \ref{algo:PDSMD} is the same as Algorithm 5 in \cite{allen2017katyusha}. 
We can take $F(\cdot) = f(\cdot) = \phi(\cdot), \psi(\cdot) = 0$ and $\sigma = 0$ in Lemma E.4. of \cite{allen2017katyusha} to get:
\begin{align*}
    &\alpha_s \langle \nabla \phi(v_{k+1}), z_k - u\rangle\\
    \leq&\frac{\alpha_s}{\tau_{1,s}} \left\{\phi(v_{k+1}) - \mathbbm{E}[\phi(y_{k+1})] + \tau_2 \phi(\widetilde{v}^s) - \tau_2 \phi(v_{k+1}) - \tau_2 \langle\nabla\phi(v_{k+1}),\widetilde{v}^s - v_{k-1}\rangle\right\}\\
    &\quad + V_{z_k}(u) - \mathbbm{E}[V_{z_{k+1}}(u)] .
\end{align*}
\end{proof}

\begin{lemma}[Coupling step 2]\label{lem:02}
Using the Lemma \ref{lem:01}, we further have
\begin{align*}
    &\alpha_s \langle \nabla \phi(v_{k+1}) , v_{k+1} - u\rangle \\
    \leq &\alpha_s \phi(v_{k+1}) + \frac{\alpha_s(1-\tau_{1,s} - \tau_2)}{\tau_{1,s}} \phi(y_k) + \frac{\alpha_s}{\tau_{1,s}} \left(  \tau_2 \phi(\widetilde{v}^s) -\mathbbm{E}[\phi(y_{k+1})] \right) + V_{z_k}(u) - \mathbbm{E}[V_{z_{k+1}}(u)] .
\end{align*}
\end{lemma}
\begin{proof}
First compute that
\begin{align*}
&\alpha_s \langle \nabla \phi(v_{k+1}) , v_{k+1} - u\rangle =  \alpha_s \langle \nabla \phi(v_{k+1}) , v_{k+1} - z_k\rangle + \alpha_s \langle \nabla \phi(v_{k+1}) , z_{k} - u\rangle\\
=& \frac{\alpha_s \tau_{2}}{\tau_{1,s}} \langle \nabla \phi(v_{k+1}) , \widetilde{v}^s - v_{k+1}\rangle +\frac{\alpha_s(1-\tau_{1,s} - \tau_2)}{\tau_{1,s}} \langle \nabla \phi(v_{k+1}) , y_k - v_{k+1}\rangle\\
&\quad + \alpha_s \langle \nabla \phi(v_{k+1}) , z_k - u\rangle\\
\leq& \frac{\alpha_s \tau_{2}}{\tau_{1,s}} \langle \nabla \phi(v_{k+1}) , \widetilde{v}^s - v_{k+1}\rangle +\frac{\alpha_s(1-\tau_{1,s} - \tau_2)}{\tau_{1,s}} (\phi(y_k) - \phi(v_{k+1})) + \alpha_s \langle \nabla \phi(v_{k+1}) , z_k - u\rangle,
\end{align*}
where the second equality by the updating rule $v_{k+1}=\tau_{1,s}z_k+\tau_2\widetilde{v}^{s}+(1-\tau_{1,s}-\tau_2)y_k$, and the inequality by convexity of $\phi$. Next, we apply Lemma \ref{lem:01} to get 
\begin{align*}
    &\alpha_s \langle \nabla \phi(v_{k+1}) , v_{k+1} - u\rangle\\
    \leq& \frac{\alpha_s \tau_{2}}{\tau_{1,s}} \langle \nabla \phi(v_{k+1}) , \widetilde{v}^s - v_{k+1}\rangle +\frac{\alpha_s(1-\tau_{1,s} - \tau_2)}{\tau_{1,s}} (\phi(y_k) - \phi(v_{k+1}))\\
    &\quad + \frac{\alpha_s}{\tau_{1,s}} \left( \phi(v_{k+1}) - \mathbbm{E}[\phi(y_{k+1})] + \tau_2 \phi(\widetilde{v}^s) - \tau_2 \phi(v_{k+1}) - \tau_2 \langle \nabla \phi(v_{k+1}),\widetilde{v}^s - v_{k+1} \rangle  \right)\\
    &\quad + V_{z_k}(u) - \mathbbm{E}[V_{z_{k+1}}(u)] \\
    = &\alpha_s \phi(v_{k+1}) + \frac{\alpha_s(1-\tau_{1,s} - \tau_2)}{\tau_{1,s}} \phi(y_k) + \frac{\alpha_s}{\tau_{1,s}} \left(  \tau_2 \phi(\widetilde{v}^s) -\mathbbm{E}[\phi(y_{k+1})] \right) + V_{z_k}(u) - \mathbbm{E}[V_{z_{k+1}}(u)] .
\end{align*}

\end{proof}

\begin{lemma}[One outer loop]\label{lem:03}
Consider the $s$th epoch, assume that all randomness in the first $s-1$ epochs are fixed, we have
\begin{align*}
    &\frac{1}{\tau_{1,s}} \sum_{k = sl}^{sl+l-1} \mathbbm{E}\langle \nabla \phi(v_{k+1}) , v_{k+1} - u\rangle + \frac{\tau_2}{\tau_{1,s+1}^2}\sum_{k = sl}^{sl+l-2} \mathbbm{E}(\phi(y_{k+1}) - \phi^*)\\
    &\quad + \frac{1-\tau_{1,s+1}}{\tau_{1,s+1}^2}\mathbbm{E}(\phi(y_{(s+1)l}) - \phi^*)\\
    \leq & \frac{1}{\tau_{1,s}} \sum_{k = sl}^{sl+l-1} \mathbbm{E}(\phi(v_{k+1}) - \phi^*)  + \frac{1-\tau_{1,s}}{\tau_{1,s}^2} \mathbbm{E}(\phi(y_{sl}) - \phi^*) + \frac{\tau_2}{\tau_{1,s}^2} \sum_{k = sl - l}^{sl-2}(\phi(y_{k+1}) - \phi^*)\\
    &\quad + 9\bar{L}(\mathbbm{E}V_{z_{sl}}(u) - \mathbbm{E}V_{z_{(s+1)l}}(u)),
\end{align*}
where $\phi^* = \min \phi(\cdot)$.
\end{lemma}
\begin{proof}
Sum up the inequality in Lemma \ref{lem:02} for $k = sl + j, j = 0,\ldots,l-1$, we have:
\begin{align}
\begin{split}\label{eq:lem3eq1}
    &\alpha_s \sum_{k = sl}^{sl+l-1} \mathbbm{E}\langle \nabla \phi(v_{k+1}) , v_{k+1} - u\rangle \\
    \leq& \sum_{k = sl}^{sl+l-1} \left\{\alpha_s \mathbbm{E}\phi(v_{k+1}) + \frac{\alpha_s(1-\tau_{1,s} - \tau_2)}{\tau_{1,s}} \mathbbm{E}\phi(y_k) + \frac{\alpha_s}{\tau_{1,s}} \left(  \tau_2 \phi(\widetilde{v}^s) -\mathbbm{E}[\phi(y_{k+1})] \right) \right.\\
    &\quad + \mathbbm{E}V_{z_k}(u) - \mathbbm{E}[V_{z_{k+1}}(u)] \bigg\}\\
    = & \sum_{k = sl}^{sl+l-1} \left\{\alpha_s \mathbbm{E}\phi(v_{k+1}) - \frac{\alpha_s(\tau_{1,s} + \tau_2)}{\tau_{1,s}} \mathbbm{E}\phi(y_{k+1}) \right\}\\
    &\quad + \frac{\alpha_s(1-\tau_{1,s} - \tau_2)}{\tau_{1,s}} [\mathbbm{E}\phi(y_{sl})-  \mathbbm{E}\phi(y_{(s+1)l})]  + \frac{\alpha_s\tau_2 l}{\tau_{1,s}} \phi(\widetilde{v}^s)+ \mathbbm{E}V_{z_{sl}}(u) - \mathbbm{E}[V_{z_{(s+1)l}}(u)].
\end{split}
\end{align}
By convexity of $\phi$, using Jensen's inequality, we have $\frac{1}{l}\sum_{k = sl - l}^{sl-1}\phi(y_{k+1})\geq \phi(\frac{1}{l}\sum_{k = sl - l}^{sl-1}y_{k+1}) = \phi(\widetilde{v}^{s})$. Thus
\begin{align*}
       &\alpha_s \sum_{k = sl}^{sl+l-1} \mathbbm{E}\langle \nabla \phi(v_{k+1}) , v_{k+1} - u\rangle + \frac{\alpha_s(\tau_{1,s} + \tau_2)}{\tau_{1,s}}\sum_{k = sl}^{sl+l-1} \mathbbm{E}\phi(y_{k+1}) \\
       \leq & \alpha_s \sum_{k = sl}^{sl+l-1} \mathbbm{E}\phi(v_{k+1})  + \frac{\alpha_s(1-\tau_{1,s} - \tau_2)}{\tau_{1,s}} [\mathbbm{E}\phi(y_{sl})-  \mathbbm{E}\phi(y_{(s+1)l})] + \frac{\alpha_s\tau_2}{\tau_{1,s}} \sum_{k = sl - l}^{sl-1}\phi(y_{k+1})\\
       &\quad + \mathbbm{E}V_{z_{sl}}(u) - \mathbbm{E}V_{z_{(s+1)l}}(u).
\end{align*}
Recall that $\alpha_s = 1/(9\bar{L}\tau_{1,s})$, we have 
\begin{align*}
       &\frac{1}{\tau_{1,s}} \sum_{k = sl}^{sl+l-1} \mathbbm{E}\langle \nabla \phi(v_{k+1}) , v_{k+1} - u\rangle + \frac{\tau_{1,s} + \tau_2}{\tau_{1,s}^2}\sum_{k = sl}^{sl+l-1} \mathbbm{E}\phi(y_{k+1}) \\
       \leq & \frac{1}{\tau_{1,s}} \sum_{k = sl}^{sl+l-1} \mathbbm{E}\phi(v_{k+1})  + \frac{1-\tau_{1,s} - \tau_2}{\tau_{1,s}^2} [\mathbbm{E}\phi(y_{sl})-  \mathbbm{E}\phi(y_{(s+1)l})] + \frac{\tau_2}{\tau_{1,s}^2} \sum_{k = sl - l}^{sl-1}\phi(y_{k+1})\\
       &\quad + 9\bar{L}(\mathbbm{E}V_{z_{sl}}(u) - \mathbbm{E}V_{z_{(s+1)l}}(u)).
\end{align*}
Deducting $\phi^* = \min \phi(\cdot)$ from both sides and rearranging terms, we get
\begin{align*}
       &\frac{1}{\tau_{1,s}} \sum_{k = sl}^{sl+l-1} \mathbbm{E}\langle \nabla \phi(v_{k+1}) , v_{k+1} - u\rangle + \frac{\tau_{1,s} + \tau_2}{\tau_{1,s}^2}\sum_{k = sl}^{sl+l-2} \mathbbm{E}(\phi(y_{k+1})-\phi^*)\\
       &\quad + \frac{1}{\tau_{1,s}^2}\mathbbm{E}(\phi(y_{(s+1)l})-\phi^*)\\
       \leq & \frac{1}{\tau_{1,s}} \sum_{k = sl}^{sl+l-1} \mathbbm{E}(\phi(v_{k+1})-\phi^*)  + \frac{1-\tau_{1,s}}{\tau_{1,s}^2} \mathbbm{E}(\phi(y_{sl})-\phi^*) + \frac{\tau_2}{\tau_{1,s}^2} \sum_{k = sl - l}^{sl-2}(\phi(y_{k+1}) - \phi^*)\\
       &\quad + 9\bar{L}(\mathbbm{E}V_{z_{sl}}(u) - \mathbbm{E}V_{z_{(s+1)l}}(u)).
\end{align*}
By our choice of $\tau_{1,s}$ and $\tau_2$, one can check
\[\frac{1-\tau_{1,s+1}}{\tau_{1,s+1}^2}\leq \frac{1}{\tau_{1,s}^2},\quad\frac{\tau_2}{\tau_{1,s+1}^2}\leq \frac{\tau_{1,s}+\tau_2}{\tau_{1,s}^2}.\]
So we further have
\begin{align*}
    &\frac{1}{\tau_{1,s}} \sum_{k = sl}^{sl+l-1} \mathbbm{E}\langle \nabla \phi(v_{k+1}) , v_{k+1} - u\rangle + \frac{\tau_2}{\tau_{1,s+1}^2}\sum_{k = sl}^{sl+l-2} \mathbbm{E}(\phi(y_{k+1}) - \phi^*) \\
    &\quad + \frac{1-\tau_{1,s+1}}{\tau_{1,s+1}^2}\mathbbm{E}(\phi(y_{(s+1)l}) - \phi^*)\\
    \leq & \frac{1}{\tau_{1,s}} \sum_{k = sl}^{sl+l-1} \mathbbm{E}(\phi(v_{k+1}) - \phi^*)  + \frac{1-\tau_{1,s}}{\tau_{1,s}^2} \mathbbm{E}(\phi(y_{sl}) - \phi^*) + \frac{\tau_2}{\tau_{1,s}^2} \sum_{k = sl - l}^{sl-2}(\phi(y_{k+1}) - \phi^*)\\
    &\quad + 9\bar{L}(\mathbbm{E}V_{z_{sl}}(u) - \mathbbm{E}V_{z_{(s+1)l}}(u)).
\end{align*}

\end{proof}

Finally, we can prove our main Theorem \ref{thm:convergence} as follows.
\begin{proof}
By Lemma \ref{lem:03}, for $s = 1,\ldots,S-1$, denote $
\delta(\cdot) := \phi(\cdot) - \phi^*$ we have:
\begin{align}
\begin{split}\label{eq:proofthmeq1}
    &\frac{1}{\tau_{1,s}} \sum_{k = sl}^{sl+l-1} \mathbbm{E}\langle \nabla \phi(v_{k+1}) , v_{k+1} - u\rangle + \frac{\tau_2}{\tau_{1,s+1}^2}\sum_{k = sl}^{sl+l-2} \mathbbm{E}\delta(y_{k+1}) + \frac{1-\tau_{1,s+1}}{\tau_{1,s+1}^2}\mathbbm{E}\delta(y_{(s+1)l})\\
    \leq & \frac{1}{\tau_{1,s}} \sum_{k = sl}^{sl+l-1} \mathbbm{E}\delta(v_{k+1})  + \frac{1-\tau_{1,s}}{\tau_{1,s}^2} \mathbbm{E}\delta(y_{sl}) + \frac{\tau_2}{\tau_{1,s}^2} \sum_{k = sl - l}^{sl-2}\mathbbm{E}\delta(y_{k+1})\\
     & \quad + 9\bar{L}(\mathbbm{E}V_{z_{sl}}(u) - \mathbbm{E}V_{z_{(s+1)l}}(u)).
\end{split}
\end{align}
For $s = 0$, apply similar proof as Lemma \ref{lem:03} on inequality \eqref{eq:lem3eq1}, we have:
\begin{align}
\begin{split}\label{eq:proofthmeq2}
    &\frac{1}{\tau_{1,0}} \sum_{k = 0}^{l-1} \mathbbm{E}\langle \nabla \phi(v_{k+1}) , v_{k+1} - u\rangle + \frac{\tau_2}{\tau_{1,1}^2}\sum_{k = 0}^{l-2} \mathbbm{E}\delta(y_{k+1}) + \frac{1-\tau_{1,1}}{\tau_{1,1}^2}\mathbbm{E}\delta(y_{l})\\
    \leq & \frac{1}{\tau_{1,0}} \sum_{k = 0}^{l-1} \mathbbm{E}\delta(v_{k+1})  + \frac{1-\tau_{1,0}-\tau_2}{\tau_{1,0}^2} \delta(y_{0}) + \frac{\tau_2 l }{\tau_{1,0}^2}\delta(\widetilde{v}^0)+ 9\bar{L}(V_{z_{0}}(u) - \mathbbm{E}V_{z_{l}}(u)).
\end{split}
\end{align}
Telescope inequality \eqref{eq:proofthmeq1} for $s = 1,\ldots,S-1$ and add inequality \eqref{eq:proofthmeq2}, we have following bound:
\begin{align}
\begin{split}\label{eq:proofthmeq3}
    &\sum_{s = 0}^{S-1}\frac{1}{\tau_{1,s}} \sum_{k = sl}^{sl+l-1} \mathbbm{E}(\langle \nabla \phi(v_{k+1}) , v_{k+1} - u\rangle - \delta(v_{k+1})) + \frac{\tau_2}{\tau_{1,S}^2}\sum_{k = (S-1)l}^{Sl-2} \mathbbm{E}\delta(y_{k+1}) \\
    \leq & \frac{1-\tau_{1,0}-\tau_2}{\tau_{1,0}^2}\delta(y_{0}) + \frac{\tau_2 l}{\tau_{1,0}^2}\delta(\widetilde{v}_{0})+ 9\bar{L}(V_{z_{0}}(u) - \mathbbm{E}V_{z_{Sl}}(u)) - \frac{1-\tau_{1,S}}{\tau_{1,S}^2}\mathbbm{E}\delta(y_{Sl}).
\end{split}
\end{align}
Now for the term $\langle \nabla \phi(v) , v - u\rangle - \phi(v)$, we note that 
\begin{align}
    \begin{split}\label{eq:proofthmeq4}
        \langle \nabla \phi(v)  , v - u\rangle - \phi(v) &=
        \langle \nabla \phi(v)  , v - u\rangle - (\langle v,b\rangle - f(x(v)) - \langle A^T v,x(v)\rangle)\\
        &=  \langle b - A x(v)  , v - u\rangle - (\langle v,b - A x(v)\rangle - f(x(v))) \\
        &=  \langle b - A x(v) , - u\rangle + f(x(v)).
    \end{split}
\end{align}
Thus
\begin{align}
    \begin{split}\label{eq:proofthmeq5}
&\sum_{s = 0}^{S-1}\frac{1}{\tau_{1,s}} \sum_{k = sl}^{sl+l-1} \mathbbm{E}(\langle \nabla \phi(v_{k+1}) , v_{k+1} - u\rangle - \delta(v_{k+1}))\\
=&\sum_{s = 0}^{S-1}\frac{1}{\tau_{1,s}} \sum_{k = sl}^{sl+l-1} \mathbbm{E}\left[ \langle b - A x(v_{k+1}) , - u\rangle + f(x(v_{k+1})) - f(x(\lambda^*))\right]\\
=&\left\langle \sum_{s = 0}^{S-1}\frac{l}{\tau_{1,s}} b - A  \mathbbm{E}\left[\sum_{s = 0}^{S-1}\frac{1}{\tau_{1,s}} \sum_{k = sl}^{sl+l-1}x(v_{k+1}) \right], - u\right\rangle \\
 & \quad + \sum_{s = 0}^{S-1}\frac{1}{\tau_{1,s}} \sum_{k = sl}^{sl+l-1} \mathbbm{E}f(x(v_{k+1})) - \sum_{s = 0}^{S-1}\frac{l}{\tau_{1,s}}f(x(\lambda^*))\\
\geq &\left\langle \sum_{s = 0}^{S-1}\frac{l}{\tau_{1,s}} b - A  \mathbbm{E}\left[\sum_{s = 0}^{S-1}\frac{1}{\tau_{1,s}} \sum_{k = sl}^{sl+l-1}x(v_{k+1}) \right], - u\right\rangle\\
& \quad + \sum_{s = 0}^{S-1}\frac{l}{\tau_{1,s}} f\left(\mathbbm{E} \left[\sum_{s = 0}^{S-1}\frac{1}{\tau_{1,s}} \sum_{k = sl}^{sl+l-1} x(v_{k+1})\right]/\sum_{s = 0}^{S-1}\frac{l}{\tau_{1,s}}\right) - \sum_{s = 0}^{S-1}\frac{l}{\tau_{1,s}}f(x(\lambda^*))\\
=& \sum_{s = 0}^{S-1}\frac{l}{\tau_{1,s}}\left[f( \mathbbm{E}(x^{S-1})) - f(x(\lambda^*)) + \langle b - A \mathbbm{E}(x^{S-1}),-u\rangle\right],
    \end{split}
\end{align}
where the inequality applies Jensen's inequality on convex function $f$. 
Plugging inequality \eqref{eq:proofthmeq5} into inequality \eqref{eq:proofthmeq3} and using the fact that $0<\frac{\tau_2}{\tau_{1,S}^2}\leq \frac{1-\tau_{1,S}}{\tau_{1,S}^2}$, we have
\begin{align}
    \begin{split}\label{eq:proofthmeq6}
    &\left(\sum_{s=0}^{S-1}\frac{l}{\tau_{1,s}}\right)(f( \mathbbm{E}(x^{S-1})) - f(x(\lambda^*))) \\
    \leq & \frac{1-\tau_{1,0}-\tau_2}{\tau_{1,0}^2}\delta(y_{0}) + \frac{\tau_2 l}{\tau_{1,0}^2}\delta(\widetilde{v}_{0})+ 9\bar{L}(V_{z_{0}}(u) - \mathbbm{E}V_{z_{Sl}}(u)) - \frac{\tau_2}{\tau_{1,S}^2}\sum_{k = (S-1)l}^{Sl-1} \mathbbm{E}\delta(y_{k+1})\\
    & + \left(\sum_{s=0}^{S-1}\frac{l}{\tau_{1,s}}\right)( \langle b - A \mathbbm{E}(x^{S-1}),u\rangle)\\
    \stackrel{}{\leq }&\frac{1-\tau_{1,0}-\tau_2}{\tau_{1,0}^2}\delta(y_{0}) + \frac{\tau_2 l}{\tau_{1,0}^2}\delta(\widetilde{v}_{0})+ 9\bar{L}V_{z_{0}}(u) - \frac{\tau_2 l}{\tau_{1,S}^2} \mathbbm{E}\delta(\widetilde{v}^S) \\
    & + \left(\sum_{s=0}^{S-1}\frac{l}{\tau_{1,s}}\right)( \langle b - A \mathbbm{E}(x^{S-1}),u\rangle),
    \end{split}
\end{align}
where the second inequality comes from the definition of $\widetilde{v}^S$ and Jensen's inequality. 
Recall that inequality \eqref{eq:proofthmeq6} holds for any $u$, including the one that minimizes the R.H.S.. 
We can further upper bound $\min_u R.H.S.$ by restricting $u\in B_H(2R):= \{u: \|u\|_{H}\leq 2R\}$:
\begin{align}
    \begin{split}\label{eq:proofthmeq7}
        &\min_{u\in B_H(2R)} 9\bar{L}V_{0}(u) + \sum_{s=0}^{S-1}\frac{l}{\tau_{1,s}}\langle b - A \mathbbm{E}(x^{S-1}),u\rangle\\
        \leq &\min_{u\in B_H(2R)} 9\bar{L}\gamma\|u\|^2_H/2 + \sum_{s=0}^{S-1}\frac{l}{\tau_{1,s}}\langle b - A \mathbbm{E}(x^{S-1}),u\rangle\\
        \leq &\min_{u\in B_H(2R)}\langle \sum_{s=0}^{S-1}\frac{l}{\tau_{1,s}}(b - A \mathbbm{E}(x^{S-1})),u\rangle + 18\bar{L}R^2\gamma\\
        =& -2R\left(\sum_{s=0}^{S-1}\frac{l}{\tau_{1,s}}\right)\| b - A \mathbbm{E}(x^{S-1})\|_{H,*} + 18\bar{L}R^2\gamma.
    \end{split}
\end{align}
Plugging the bound \eqref{eq:proofthmeq7} into inequality \eqref{eq:proofthmeq6}, we have
\begin{align}
    \begin{split}\label{eq:proofthmeq8}
    &\left(\sum_{s=0}^{S-1}\frac{l}{\tau_{1,s}}\right)(f( \mathbbm{E}(x^{S-1})) - f(x(\lambda^*))) + \frac{\tau_2 l}{\tau_{1,S}^2} \mathbbm{E}\delta(\widetilde{v}^S)+  2R\left(\sum_{s=0}^{S-1}\frac{l}{\tau_{1,s}}\right)\| b - A \mathbbm{E}(x^{S-1})\|_{H,*}\\
    \leq & \frac{1-\tau_{1,0}-\tau_2}{\tau_{1,0}^2}\delta(y_{0}) + \frac{\tau_2 l}{\tau_{1,0}^2}\delta(\widetilde{v}_{0})  + 18\bar{L}R^2\gamma \\
    =&2l\delta(0)+ 18\bar{L}R^2\gamma.
%     &\frac{\tau_2 l}{\tau_{1,S}^2} (\mathbbm{E}\phi(\widetilde{v}^S) + f( \mathbbm{E}(x^{S-1}))) \\
%     \leq& \frac{1-\tau_{1,0}-\tau_2}{\tau_{1,0}^2}\delta(y_{0}) + \frac{\tau_2 l}{\tau_{1,0}^2}\delta(\widetilde{v}_{0})+ 9\bar{L}(V_{z_{0}}(u) - \mathbbm{E}V_{z_{Sl}}(u)) + \sum_{s=0}^{S-1}\frac{l}{\tau_{1,s}}\langle b - A \mathbbm{E}(x^{S-1}),u\rangle\\
%     \leq &\frac{1-\tau_{1,0}-\tau_2}{\tau_{1,0}^2}\delta(y_{0}) + \frac{\tau_2 l}{\tau_{1,0}^2}\delta(\widetilde{v}_{0})+ 9\bar{L}V_{z_{0}}(u) + \sum_{s=0}^{S-1}\frac{l}{\tau_{1,s}}\langle b - A \mathbbm{E}(x^{S-1}),u\rangle\\
%     = &\frac{1-\tau_{1,0}-\tau_2}{\tau_{1,0}^2}\delta(0) + \frac{\tau_2 l}{\tau_{1,0}^2}\delta(0)+ 9\bar{L}V_{0}(u) + \sum_{s=0}^{S-1}\frac{l}{\tau_{1,s}}\langle b - A (x^{S-1}),u\rangle.
    \end{split}
\end{align}
Calculate $\sum_{s=0}^{S-1}\frac{1}{\tau_{1,s}} = \sum_{s=0}^{S-1}(s+4)/2 = (2S+3)S/4\geq S^2/2$, then
\begin{align}
    \label{eq:proofthmeq9}
   f( \mathbbm{E}(x^{S-1}))-f(x^*) \leq \frac{4}{S^2 l}\left[ l\delta(0)+ 9\bar{L}R^2\gamma\right] .
\end{align}
On the other hand, notice that $f(x(\lambda^*)) = -\phi(\lambda^*):=\phi^*$ and
\begin{align}
\begin{split}\label{eq:proofthmeq10}
    & f(\mathbbm{E}(x^{S-1})) - f(x(\lambda^*))\\
    =& f( \mathbbm{E}(x^{S-1})) + \phi^*\\
    =& f(\mathbbm{E}(x^{S-1})) + \langle \lambda^* ,b \rangle + \max_x( -f(x) - \langle A^T \lambda^* ,x\rangle)\\
    \geq & f(\mathbbm{E}(x^{S-1})) + \langle \lambda^* ,b \rangle -f(\mathbbm{E}(x^{S-1})) - \langle A^T \lambda^* ,\mathbbm{E}(x^{S-1})\rangle\\
    =&\langle \lambda^* ,b -A\mathbbm{E}(x^{S-1})\rangle \geq - R\| \mathbbm{E}[b -Ax^{S-1}]\|_{H,*}.
\end{split}
\end{align}
Plugging inequality \eqref{eq:proofthmeq10} into inequality \eqref{eq:proofthmeq8}, we have
\begin{align}
    \begin{split}\label{eq:proofthmeq11}
        & R\left(\sum_{s=0}^{S-1}\frac{l}{\tau_{1,s}}\right)\| \mathbbm{E}[ b - Ax^{S-1}]\|_{H,*}\leq  2l\delta(0)+ 18\bar{L}R^2\gamma.
    \end{split}
\end{align}
Thus 
\begin{align}
    \begin{split}\label{eq:proofthmeq12}
        \| \mathbbm{E}[ b - Ax^{S-1}]\|_{H,*}\leq\frac{4\left[ l\delta(0)+ 9\bar{L}R^2\gamma\right]}{S^2 l R}.
    \end{split}
\end{align}
Further check that 
\begin{align}
    \begin{split}\label{eq:proofthmeq13}
    \delta(0) = \phi(0)-\phi^*\leq \langle \nabla \phi(\lambda^{\ast}),0-\lambda^{\ast}\rangle +\frac{\bar{L}}{2}\Vert0-\lambda^{\ast} \Vert_H^2=\frac{\bar{L}}{2}\Vert \lambda^{\ast} \Vert_H^2\leq \frac{\bar{L}}{2} R^2.
    \end{split}
\end{align}
Plugging the bound \eqref{eq:proofthmeq13} into inequalities \eqref{eq:proofthmeq9} and \eqref{eq:proofthmeq12}, we get the theorem claim.
\end{proof}

\section{Proof for Lemma \ref{lem:smooth}}\label{app:B}
\begin{proof}
By Proposition 2 of \cite{xie2022}, $\phi_i(\cdot)$ is $\frac{np_i'}{\eta}$ smooth w.r.t. $\|\cdot\|_2$. 
So here we only show the second part of the statement. 
That is, prove the smoothness w.r.t. $\|\cdot\|_{\infty}$.

By 
\[\phi_i(\lambda) = np_i'\left(- \langle q', \lambda\rangle - \eta\log p'_i + \eta\log \left(\sum_{j=1}^n \exp((\lambda_j - c_{i,j} - \eta)/\eta)\right) + \eta\right),\]
we calculate that 
\[\nabla \phi_i(\lambda) = np_i'\left(- q' + \frac{(\exp((\lambda_k - c_{i,k})/\eta))_{k=1,\ldots,n}}{ \left(\sum_{j=1}^n \exp((\lambda_j - c_{i,j})/\eta)\right)}\right).\]
The goal is $\forall \lambda,\lambda'$, bound the $\|\cdot\|_{\infty,*} = \|\cdot\|_1$ of following difference in the gradient:
\[\nabla \phi_i(\lambda) - \nabla \phi_i(\lambda') = np_i'\left( \frac{(\exp((\lambda_k - c_{i,k})/\eta))_{k=1,\ldots,n}}{ \left(\sum_{j=1}^n \exp((\lambda_j - c_{i,j})/\eta)\right)} - \frac{(\exp((\lambda_k' - c_{i,k})/\eta))_{k=1,\ldots,n}}{ \left(\sum_{j=1}^n \exp((\lambda_j' - c_{i,j})/\eta)\right)}\right).\]
Further denote $\Delta \lambda = \lambda' - \lambda$, then
\begin{align*}
    &\|\nabla \phi_i(\lambda) - \nabla \phi_i(\lambda')\|_{1} \\
    =& np_i'\left\| \frac{(\exp((\lambda_k - c_{i,k})/\eta))_{k=1,\ldots,n}}{ \left(\sum_{j=1}^n \exp((\lambda_j - c_{i,j})/\eta)\right)} - \frac{(\exp((\lambda_k + (\Delta\lambda)_k - c_{i,k})/\eta))_{k=1,\ldots,n}}{ \left(\sum_{j=1}^n \exp((\lambda_j + (\Delta\lambda)_j - c_{i,j})/\eta)\right)}\right\|_1\\
    =& np_i'\left\| \frac{(\exp((\lambda_k - c_{i,k})/\eta))_{k=1,\ldots,n}}{ \left(\sum_{j=1}^n \exp((\lambda_j - c_{i,j})/\eta)\right)} - \frac{(\exp((\lambda_k - c_{i,k})/\eta) * \exp((\Delta\lambda/\eta)_k))_{k=1,\ldots,n}}{ \left(\sum_{j=1}^n \exp((\lambda_j - c_{i,j})/\eta)\exp((\Delta\lambda/\eta)_j)\right)}\right\|_1.
\end{align*}
Taking $\bm{a} = (\exp((\lambda_k - c_{i,k})/\eta))_{k=1,\ldots,n}$ and $\bm{b} = \Delta\lambda/\eta$ in Lemma \ref{lem:unit_bound}, we immediately have
\begin{align}
    \|\nabla \phi_i(\lambda) - \nabla \phi_i(\lambda')\|_{1}\leq n p_i' 5 \|\Delta\lambda/\eta\|_{\infty} = \frac{5np_i'}{\eta}\|\lambda - \lambda'\|_{\infty}.
\end{align}
Thus, $\phi_i(\cdot)$ is $\frac{5np_i'}{\eta}$ smooth w.r.t. $\|\cdot\|_{\infty}$ norm.
\end{proof}

The following lemma is used in the proof of Lemma \ref{lem:smooth}. 
\begin{lemma}\label{lem:unit_bound}
Consider two vectors $\bm{a}, \bm{b}\in \mathbb{R}^d$, and let $\exp(\bm{b})$ be the element-wise exponential of $\bm{b}$. When $\bm{a} > \bm{0}$, we have
\[\left\|\frac{\bm{a}}{\|\bm{a}\|_1} - \frac{\bm{a}\circ\exp(\bm{b})}{\|\bm{a}\circ\exp(\bm{b})\|_1}\right\|_1 \leq 5\|\bm{b}\|_{\infty}.\]
\end{lemma}

\begin{proof}
Consider two cases:

First, when $\|\bm{b}\|_{\infty} > 0.5$:
\[\left\|\frac{\bm{a}}{\|\bm{a}\|_1} - \frac{\bm{a}\circ\exp(\bm{b})}{\|\bm{a}\circ\exp(\bm{b})\|_1}\right\|_1 \leq \left\|\frac{\bm{a}}{\|\bm{a}\|_1}\right\|_1 + \left\|\frac{\bm{a}\circ\exp(\bm{b})}{\|\bm{a}\circ\exp(\bm{b})\|_1}\right\|_1 = 2 < 5 *0.5 < 5\|\bm{b}\|_{\infty}.\]

Second, when $\|\bm{b}\|_{\infty} \leq 0.5$:
\begin{align*}
    \begin{split}
        &\left\|\frac{\bm{a}}{\|\bm{a}\|_1} - \frac{\bm{a}\circ\exp(\bm{b})}{\|\bm{a}\circ\exp(\bm{b})\|_1}\right\|_1\\
        =&\sum_{i=1}^d \frac{|\sum_{j = 1}^d a_i a_j \exp(b_j) - \sum_{j=1}^d a_i \exp(b_i) a_j|}{\|\bm{a}\|_1\|\bm{a}\circ\exp(\bm{b})\|_1}\\
        \leq& \sum_{i=1}^d \frac{\sum_{j = 1}^d a_i a_j | \exp(b_j) - \exp(b_i)|}{\|\bm{a}\|_1\|\bm{a}\circ\exp(\bm{b})\|_1}\\
        \leq & \sum_{i=1}^d \frac{\sum_{j = 1}^d a_i a_j (| \exp(b_j) -1 | + |\exp(b_i)-1|)}{\|\bm{a}\|_1\|\bm{a}\|_1 \exp(-0.5)}\\
        \leq & 2\exp(0.5)(\exp(0.5) - 1)\sum_{i=1}^d \frac{\sum_{j = 1}^d a_i a_j (|b_j| + |b_i|)}{\|\bm{a}\|_1\|\bm{a}\|_1}\\
        \leq& 4\exp(0.5)(\exp(0.5) - 1)\sum_{i=1}^d \frac{\sum_{j = 1}^d a_i a_j \|\bm{b}\|_{\infty}}{\|\bm{a}\|_1\|\bm{a}\|_1}\\
        = & 4\exp(0.5)(\exp(0.5) - 1) \|\bm{b}\|_{\infty} < 5\|\bm{b}\|_{\infty}.
    \end{split}
\end{align*}
Combining two cases, we have the lemma holds. 
\end{proof}

\section{Proof for Theorem \ref{thm:computation_ot}}\label{app:C}
The full procedure for finding an $\epsilon$-solution to OT using PDASMD is given in the following algorithm:
\begin{algorithm}[H]
   \caption{Approximating OT by \textsc{PDASMD}}
   \begin{algorithmic}
   \label{algo:2}
   \STATE \textbf{Input}: Accuracy $\epsilon>0$, $\eta=\frac{\epsilon}{4\log(n)}$ and $\epsilon'=\frac{\epsilon}{8\Vert C\Vert_{\infty}}$.
   \STATE \textbf{Step 1}: Let $\bm{p}'\in \Delta_n$ and $\bm{q}'\in \Delta_n$ be 
    \[ \begin{pmatrix}\bm{p}'\\ \bm{q}'\end{pmatrix}
   =\left(1-\frac{\epsilon'}{8}\right)\begin{pmatrix}\bm{p}\\\bm{q}\end{pmatrix}+\frac{\epsilon'}{8n}\begin{pmatrix}\textbf{1}_n\\\textbf{1}_n\end{pmatrix}.\] 
   
   \STATE \textbf{Step 2}: Compute $\widetilde{\bm{X}}$ by PDASMD on objective \eqref{eq:OT_pen} long enough such that $f(\mathbbm{E}\widetilde{\bm{x}})-f(\bm{x}^{\ast})\leq \frac{\epsilon}{4}$ and $\Vert A\mathbbm{E}\widetilde{\bm{x}}-{b}\Vert_{1}\leq \frac{\epsilon'}{2}$.
   \STATE \textbf{Step 3}: Round $\widetilde{X}$ to $\widehat{X}$ by Algorithm 2 in \cite{altschuler2017near} such that $\widehat{X}\textbf{1}_n=\bm{p}$, $\widehat{X}^T\textbf{1}_n=\bm{q}$.
   \STATE \textbf{Output}: $\widehat{X}$
   \end{algorithmic}
   \end{algorithm}

The total computational cost of Algorithm \ref{algo:2} is given in Theorem \ref{thm:computation_ot}, and we prove it as follows:

\begin{proof}
We have the convergence result in Theorem \ref{thm:convergence} holds for the dual formulation. 
To extend the proof of Theorem \ref{thm:convergence} to the semi-dual formulation for the OT problem, we just need the following equality to hold:
\[[Ax(v)-b]=\begin{bmatrix}
     \mathbf{0}_n  \\
     \nabla \phi(v)
\end{bmatrix}.\]
One can easily check it is true for the semi-dual of OT.
Moreover, paper \cite{xie2022} shows that the stopping criteria in step 2 of Algorithm \ref{algo:2} guarantees 
\[\mathbbm{E}\langle C,\widehat{X}\rangle \leq\langle C,X^*\rangle +\epsilon.\]
That is, the output of Algorithm \ref{algo:2} is an $\epsilon-$solution.
We now focus on the computational complexity of Algorithm \ref{algo:2}.

\textbf{Case 1: $\|\cdot\|_H = \|\cdot\|_2$}. By Theorem \ref{thm:convergence} we have
\begin{align}
    &\| \mathbbm{E}[ b - A\widetilde{x}]\|_{2}\leq\frac{2\left[ l\bar{L} R + 18\bar{L}R\gamma\right]}{S^2 l},\\
    & f( \mathbbm{E}(\widetilde{x}))-f(x^*) \leq \frac{2}{S^2 l}\left[ l\bar{L} R^2 + 18\bar{L}R^2\gamma\right], 
\end{align}
where $\bar{L} = \frac{1}{n} \sum_{i=1}^{n}\frac{np_i'}{\eta} = \frac{1}{\eta}$, and $R$ is an upper bound for $\|\lambda^*\|_2$. By Lemma 3.2 in \cite{lin2019efficient}, $R = \eta\sqrt{n}(R' + .5)$ for $R' = \Vert C\Vert_{\infty}/\eta+\log (n)-2\log(\min_{1\leq i,j\leq n} \{p'_i,q'_j\}) \leq 4\Vert C\Vert_{\infty}\log(n)/\epsilon+\log (n)-2\log (\epsilon) + 2\log(64n\|C\|_{\infty}) = \mathcal{O}(\Vert C\Vert_{\infty}\log(n)/\epsilon)$. 

Using the bound $\Vert A\mathbbm{E}\widetilde{x}-{b}\Vert_1 \leq \sqrt{2n}\Vert A\mathbbm{E}\widetilde{x}-{b}\Vert_2$ we have the stopping criteria in step 2 satisfied for 
\begin{align*}
    \begin{split}
        S =& \max\left\{\mathcal{O}\left(\sqrt{\frac{\bar{L} R\sqrt{n}}{\epsilon'}}\right), \mathcal{O}\left(\sqrt{\frac{\bar{L}\gamma R\sqrt{n}}{l\epsilon'}}\right), \mathcal{O}\left(\sqrt{\frac{\bar{L} R^2}{\epsilon}}\right), \mathcal{O}\left(\sqrt{\frac{\bar{L}R^2\gamma}{l\epsilon}}\right)\right\}\\
        =&\max\left\{\mathcal{O}\left(\sqrt{\frac{{n}\log(n)\|C\|^2_{\infty}}{\epsilon^2}}\right), \mathcal{O}\left(\sqrt{\frac{\log(n)\gamma n\|C\|^2_{\infty}}{l\epsilon^2}}\right),\right.\\
        &\left.\quad\quad\quad\quad\quad\quad\mathcal{O}\left(\frac{{n}^{.5}\|C\|_{\infty}\sqrt{\log n}}{\epsilon}\right), \mathcal{O}\left(\sqrt{\frac{n\log(n) \|C\|^2_{\infty}\gamma}{l\epsilon^2}}\right)\right\}\\
        =&\mathcal{O}\left(\frac{n^{.5}}{\epsilon}\max\left(\|C\|_{\infty}\sqrt{\log n}, \sqrt{\log(n) \|C\|^2_{\infty}\gamma/l}\right)\right)\\
        =& \mathcal{O}\left(\frac{n^{.5}\|C\|_{\infty}}{\epsilon}\left(\sqrt{\log n}+\sqrt{\log(n)\gamma/l}\right)\right).
    \end{split}
\end{align*}
In Algorithm \ref{algo:2}, step 1 and step 3 has a total number of $\mathcal{O}(n^2)$ operations, and the algorithm complexity is dominated by step 2. 
Now each outer loop of PDASMD has $\mathcal{O}(n^2 + nl)$ operations. 
Thus the total number of arithmetic operations of Algorithm \ref{algo:2} is
\[\mathcal{O}\left(\frac{n^{1.5}\|C\|_{\infty}}{\epsilon}(n+l)\left(\sqrt{\log n}+\sqrt{\log(n) \gamma/l}\right)\right) = \mathcal{\widetilde{O}}\left(\frac{n^{2.5}\|C\|_{\infty}(1+\sqrt{\gamma/n})}{\epsilon}\right) .\]

\textbf{Case 2: $\|\cdot\|_H = \|\cdot\|_\infty$}. Then $\|\cdot\|_{H,*} = \|\cdot\|_1$, and Theorem \ref{thm:convergence} implies that 
\begin{align}
    &\| \mathbbm{E}[ b - A\widetilde{x}]\|_{1}\leq\frac{2\left[ l\bar{L} R + 18\bar{L}R\gamma\right]}{S^2 l},\\
    & f( \mathbbm{E}(\widetilde{x}))-f(x^*) \leq \frac{2}{S^2 l}\left[ l\bar{L} R^2 + 18\bar{L}R^2\gamma\right] ,
\end{align}
where $\bar{L} = \frac{5}{\eta}$. Now $R$ is an upper bound for $\|\lambda^*\|_\infty$, and again by Lemma 3.2 in \cite{lin2019efficient}, $R = \eta (R' + .5)$ for $R' = \mathcal{O}(\Vert C\Vert_{\infty}\log(n)/\epsilon)$. 

Thus the stopping criteria in step 2 of Algorithm \ref{algo:2} is satisfied for 
\begin{align*}
    \begin{split}
        S = & \max\left\{\mathcal{O}\left(\sqrt{\frac{\bar{L} R}{\epsilon'}}\right), \mathcal{O}\left(\sqrt{\frac{\bar{L}R\gamma}{l\epsilon'}}\right), \mathcal{O}\left(\sqrt{\frac{\bar{L} R^2}{\epsilon}}\right), \mathcal{O}\left(\sqrt{\frac{\bar{L}R^2\gamma}{l\epsilon}}\right)\right\}\\
        = & \max\left\{\mathcal{O}\left(\sqrt{\frac{\log(n)\|C\|^2_{\infty}}{\epsilon^2}}\right), \mathcal{O}\left(\sqrt{\frac{\log(n)\gamma\|C\|^2_{\infty}}{l\epsilon^2}}\right),\right.\\
        &\left.\quad\quad\quad\quad\quad\quad\mathcal{O}\left(\sqrt{\frac{\|C\|^2_{\infty}\log n}{\epsilon^2}}\right), \mathcal{O}\left(\sqrt{\frac{\log(n) \|C\|^2_{\infty}\gamma}{l\epsilon^2}}\right)\right\}\\
        = & \mathcal{O}\left(\frac{\|C\|_{\infty}}{\epsilon}\max\left(\sqrt{\log n}, \sqrt{\log(n)\gamma/n}\right)\right)\\
        = & \mathcal{O}\left(\frac{\|C\|_{\infty}}{\epsilon}\left(\sqrt{\log n}+\sqrt{\log(n)\gamma/n}\right)\right) .
    \end{split}
\end{align*}
%In Algorithm \ref{algo:2}, step 1 and step 3 has total number of $\mathcal{O}(n^2)$ operations, and the algorithm complexity is dominated by step 2.
Now each outer loop of PDASMD has $\mathcal{O}(n^2 + nl) = \mathcal{O}(n^2)$ operations; thus the total number of arithmetic operations of Algorithm \ref{algo:2} is
\[\mathcal{O}\left(\frac{n^{2}\|C\|_{\infty}}{\epsilon}\left(\sqrt{\log n}+\sqrt{\log(n)\gamma/n}\right)\right) = \mathcal{\widetilde{O}}\left(\frac{n^{2}\|C\|_{\infty}(1 +\sqrt{\gamma/n})}{\epsilon}\right). \]
\end{proof}

\section{Stochastic Sinkhorn Algorithm and Proof of Computational Complexity}\label{app:E}
We first describe the Stochastic Sinkhorn algorithm.
In the Stochastic Sinkhorn algorithm, the following definitions are used:
\begin{definition}[Increasing probability function]
An increasing probability function $\Psi: \mathbb{R}^p_+ \to \Delta_p$ is such that
\[\Psi(\bm{h}) = \left(\frac{g(h_i)}{\sum_i {g(h_i)}}\right)_i,\]
where $g: \mathbb{R}_+ \to \mathbb{R}_+$ is an increasing positive function.
\end{definition}

\begin{definition}[KL violation]
For a matrix $M\in \mathbb{R}_+^{p\times p}$ and two vectors $\bm{p},\bm{q}\in \Delta_p$, define the KL violation
\begin{align}
    \rho(M;\bm{p},\bm{q}) = \left[\begin{array}{l}
         (\mathcal{KL}(p_i\Vert (M \mathbf{1})_i))_{i=1,\ldots,p}\\
         (\mathcal{KL}(q_j\Vert (M^T \mathbf{1})_j))_{j=1,\ldots,p}
    \end{array}\right].
\end{align}
\end{definition}
The Stochastic Sinkhorn algorithm for solving problem \eqref{eq:OT_pen} is as following:

\begin{algorithm}[H]
\caption{Stochastic Sinkhorn}
\begin{algorithmic}
  \label{algo:3}
   \STATE \textbf{Input}: $C,\bm{p}',\bm{q}',\Psi$,$\eta$
   \STATE Calculate $A = \exp(- C/\eta)$ where all the operations are element-wise;
   \STATE \textbf{Initialize}: $\bm{u}^0 = \bm{v}^0 = \mathbf{1}$;
   \FOR{k=0,\ldots,K-1}
   \STATE Calculate $\bm{h} = \Psi(\rho(X(\bm{u}^k,\bm{v}^k);\bm{p}',\bm{q}'))$, where $X(\bm{u}^k,\bm{v}^k) = diag(\bm{u}^k)Adiag(\bm{v}^k)$\;
   \STATE Sample index $I$ with
\[P(I = i) = h_i, \forall i \in \{1,2,...,2n\}.\]
    \IF{$I\leq n$}
    \STATE $\bm{u}^{k+1} =(u^k_1,\ldots,u^k_{I-1},    p'_I/(A\bm{v}^k)_I, u^k_{I+1},\ldots,u^k_{n})^T$, $\bm{v}^{k+1} = \bm{v}^k;$
   \ELSE
    \STATE $\bm{u}^{k+1} = \bm{u}^k$, $\bm{v}^{k+1} =(v^k_1,\ldots,v^k_{I-n-1},    q'_{I-n}/(A^T\bm{u}^k)_{I-n}, v^k_{I-n+1},\ldots,v^k_{n})^T.$
   \ENDIF
   \ENDFOR
   \STATE \textbf{Output}: $\widetilde{X} = diag(\bm{u}^K)A diag(\bm{v}^K)$.
\end{algorithmic}
\end{algorithm}
To find a $\epsilon$-solution to OT, an extra rounding step is required. 
The full procedure is given in Algorithm \ref{algo:4}.

\begin{algorithm}[H]
   \caption{Approximating OT by Stochastic Sinkhorn}
   \begin{algorithmic}
   \label{algo:4}
   \STATE \textbf{Input}: Accuracy $\epsilon>0$, $\eta=\frac{\epsilon}{4\log(n)}$ and $\epsilon'=\frac{\epsilon}{8\Vert C\Vert_{\infty}}$. 
   \STATE \textbf{Step 1}: Let $\bm{p}'\in \Delta_n$ and $\bm{q}'\in \Delta_n$ be 
    \[ \begin{pmatrix}\bm{p}'\\ \bm{q}'\end{pmatrix}
   =\left(1-\frac{\epsilon'}{8}\right)\begin{pmatrix}\bm{p}\\\bm{q}\end{pmatrix}+\frac{\epsilon'}{8n}\begin{pmatrix}\textbf{1}_n\\\textbf{1}_n\end{pmatrix}.\] 
   
   \STATE \textbf{Step 2}: Compute $\widetilde{X}$ by Stochastic Sinkhorn until $\|\widetilde{X}\mathbf{1} - \bm{p}'\|_1 + \|\widetilde{X}^T\mathbf{1} - \bm{q}'\|_1\leq \frac{\epsilon'}{2}$.
   \STATE \textbf{Step 3}: Round $\widetilde{X}$ to $\widehat{X}$ by Algorithm 2 in \cite{altschuler2017near} such that $\widehat{X}\textbf{1}_n=\bm{p}$, $\widehat{X}^T\textbf{1}_n=\bm{q}$.
   \STATE \textbf{Output}: $\widehat{X}$. 
   \end{algorithmic}
   \end{algorithm}  

We now prove the computational complexity of Stochastic Sinkhorn in Theorem \ref{thm:stoc_sinkhorn}. 
To prove it, we first need the convergence of Algorithm \ref{algo:3}, which we show in the following Lemma.
\begin{lemma}\label{lem:04}
For a given $\epsilon>0$, we have that Algorithm \ref{algo:3} returns a matrix $\widetilde{X}$ such that \[\mathbbm{E}[\|\widetilde{X}\mathbf{1} - p'\|_1 + \|\widetilde{X}^T\mathbf{1} - q'\|_1]\leq \epsilon\]
in the number of iterations 
$$k\leq 2+ 112n R/\epsilon .$$
\end{lemma}

\begin{proof}
Denote $(x^k,y^k):= (\log u^k, \log v^k)$ and the dual function $f(x,y) = \sum_{i,j} A_{i,j} \exp(x_i + y_j) - \langle p', x\rangle - \langle q', y\rangle$. 
Denote 
$E_k = \mathbbm{E}[\|X(u^k,v^k)\mathbf{1} - p'\|_1 + \|X(u^k,v^k)^T\mathbf{1} - q'\|_1]$. 
By (21) in \cite{abid2018stochastic}, Algorithm \ref{algo:3} has
\begin{equation}\label{eq:lem4eq1}
    E[f(x^k,y^k) - f(x^{k+1},y^{k+1})] > \frac{E_k^2}{28n}.
\end{equation}
Since Algorithm \ref{algo:3} only updates one element in $u$ or $v$, and the updating rule for that element is the same as Greenkhorn, we have that Corollary 3.3 in \cite{lin2019efficient} holds.
Adding expectations to both sides, we get:
\begin{equation}\label{eq:lem4eq2}
    E[f(x^k,y^k) - f(x^*,y^*)]\leq 4RE_k.
\end{equation}
Let $\delta_k = E[f(x^k,y^k) - f(x^*,y^*)]$, then by inequalities \eqref{eq:lem4eq1} and \eqref{eq:lem4eq2} we have
\begin{equation}
    \delta_k - \delta_{k+1} \stackrel{\eqref{eq:lem4eq1}}{\geq}\frac{E_k^2}{28n}\stackrel{\eqref{eq:lem4eq2}}{\geq}\frac{\delta_k^2}{448n R^2} . 
\end{equation}
That is,
\begin{equation}
    \delta_k - \delta_{k+1} \geq \max\left\{\frac{\epsilon^2}{28n},\frac{\delta_k^2}{448n R^2}\right\}.
\end{equation}
We adopt the strategy in \cite{dvurechensky2018computational} to split the process of $\{\delta_k\}$ into two halves:

First, consider the process from $\delta_1$ to $\delta_{t}$:
\[\frac{\delta_{t}}{448nR^2} \leq \frac{\delta_{t-1}}{448nR^2} - \left(\frac{\delta_{t-1}}{448nR^2}\right)^2\leq \frac{1}{t - 1+ 448nR^2/\delta_1} \Rightarrow t\leq 1 + \frac{448nR^2}{\delta_t} - \frac{448nR^2}{\delta_1}.\]
Second, consider the process from $\delta_{t}$ to $\delta_{t+m}$:
\[\delta_{t+m} \leq \delta_t - \frac{\epsilon^2 m}{28n} \Rightarrow m\leq\frac{28n(\delta_t - \delta_{t+m})}{\epsilon^2}.\]
So the total number of iterations $k = t + m$ can be optimized over $\delta_t$, i.e.
\[
k \leq \min_{\delta_t \in (0,\delta_1]} \left( 2 + \frac{448nR^2}{\delta_t} - \frac{448nR^2}{\delta_1} +\frac{28n\delta_t}{\epsilon^2}\right) \leq 2 + \frac{112nR}{\epsilon}.
\]
\end{proof}
Then we can prove Theorem \ref{thm:stoc_sinkhorn} as follows:
\begin{proof}
By Lemma \ref{lem:04}, we have $\mathbbm{E}[\|\widetilde{X}\mathbf{1} - p'\|_1 + \|\widetilde{X}^T\mathbf{1} - q'\|_1]\leq \epsilon'/2$
for the number of iterations $k = 2+ 224n R/\epsilon'$. 
Thus for this $k$, we also have
\[
\mathbbm{E}[\|\widetilde{X}\mathbf{1} - p\|_1 + \|\widetilde{X}^T\mathbf{1} - q\|_1]\leq\mathbbm{E}[\|\widetilde{X}\mathbf{1} - p'\|_1 + \|\widetilde{X}^T\mathbf{1} - q'\|_1] + \|p - p'\|_1 + \|q - q'\|_2\leq \epsilon'.
\]
By Theorem 1 in \cite{altschuler2017near}, 
\[\mathbbm{E}\langle C,\widehat{X}\rangle - \langle C,X^*\rangle \leq \epsilon/2 + 4(\mathbbm{E}[\|\widetilde{X}\mathbf{1} - p\|_1 + \|\widetilde{X}^T\mathbf{1} - q\|_1]) \|C\|_{\infty}\leq \epsilon,\]
which is an $\epsilon-$solution.

Now calculate the number of arithmetic operations, step 1 requires $\mathcal{O}(n)$; 
step 2 requires $k$ iterations of Algorithm \ref{algo:3}, each iteration requires $\mathcal{O}(n)$ operation \cite{abid2018stochastic}; 
step 3 requires $\mathcal{O}(n^2)$ operations \cite{altschuler2017near}. 
Thus the total number of operations is 
\[\mathcal{O}(n^2 R/\epsilon') = \mathcal{O}(n^2 \|C\|^2_{\infty}\log n /\epsilon^2).\]
\end{proof}

\section{PDASMD-B Algorithm and the Convergence Rate}\label{app:D}
%The pseudo-code of PDASMD-B algorithm is as follows:

In this Section, we prove the convergence of PDASMD-B. 
The convergence rate of PDASMD-B is in the following theorem:

\begin{theorem}[Convergence of PDASMD-B]\label{thm:convergence_batch}
In Algorithm \ref{algo:PDSMD-B}, assume that the dual optimal solution has $\|\lambda^*\|_{H}\leq R$. 
Then we have the convergence of Algorithm \ref{algo:PDSMD-B} as:
\begin{align}
    &\| \mathbbm{E}[ \bm{b} - A\bm{x}^{S-1}]\|_{H,*}\leq\frac{2\left[ (1+ (l-1)/B)\bar{L} R + 18\bar{L}R\gamma\right]}{S^2 l},\\
    & f( \mathbbm{E}(\bm{x}^{S-1}))-f(\bm{x}^*) \leq \frac{2}{S^2 l}\left[(1+ (l-1)/B)\bar{L} R^2 + 18\bar{L}R^2\gamma\right] .
\end{align}
%In Algorithm \ref{algo:PDSMD-B}, assume that the distance generating function used for $V_{x}(y)$ is $\gamma-$smooth w.r.t $\|\cdot\|_{H}$ norm, and the dual optimal solution is such that $\|\lambda^*\|_{H}\leq R$. 
%We have the convergence of the algorithm as:
%\begin{align}
%    &\| \mathbbm{E}[ b - Ax^{S-1}]\|_{H,*}\leq\frac{8\left[ l\bar{L} R + 18\bar{L}R\gamma\right]}{S^2 l},\\
%    & f( \mathbbm{E}(x^{S-1}))-f(x^*) \leq \frac{8}{(S+4)^2 l}\left[ l\bar{L} R^2 + 18\bar{L}R^2\gamma\right] .
%\end{align}
\end{theorem}

The key to proving Theorem \ref{thm:convergence_batch} is to find an analog to Lemma \ref{lem:01}, which we do in the following steps. 

\begin{lemma}[Variance upper bound]
\begin{equation}\label{eq:lem6}
\mathbbm{E}[\|\widetilde{\nabla}_{k+1} - \nabla \phi(v_{k+1})\|_{H,*}^2 ]\leq \frac{8\bar{L}}{B}(\phi(\widetilde{v}^s) - \phi(v_{k+1}) - \langle \nabla\phi(v_{k+1}),\widetilde{v}^s-v_{k+1}\rangle).
\end{equation}
\end{lemma}
\begin{proof}
Each $\phi_i$ is convex and $L_i$-smooth, then by Theorem 2.1.5. in \cite{nesterov2003introductory} we have
\begin{equation}\label{eq:lem6eq1}
    \|\nabla\phi_i(v_{k+1}) - \nabla\phi_i(\widetilde{v}^s)\|_{H,*}^2 \leq 2L_i(\phi_i(\widetilde{v}^s) - \phi_i(v_{k+1}) - \langle \nabla\phi_i(v_{k+1}),\widetilde{v}^s-v_{k+1}\rangle).
\end{equation}
Take expectation with respect to the randomness of index set $I$, note that all indexes in $I$ are independently selected, we have
\begin{align*}
    &\mathbbm{E}[\|\widetilde{\nabla}_{k+1} - \nabla \phi(v_{k+1})\|_{H,*}^2 ]\\
    =&\frac{1}{B} \mathbbm{E}\left[\left\| \nabla\phi(\widetilde{v}^s) + \frac{1}{mp_i}(\nabla \phi_{i}(v_{k+1})-\nabla \phi_{i}(\widetilde{v}^{s})) - \nabla \phi(v_{k+1})\right\|_{H,*}^2 \right]\\
    \leq & \frac{1}{B} \mathbbm{E}\left[2\left\| \frac{1}{mp_i}(\nabla \phi_{i}(\widetilde{v}^{s})-\nabla \phi_{i}(v_{k+1}))\right\|_{H,*}^2 + 2\|\nabla\phi(\widetilde{v}^s) - \nabla \phi(v_{k+1})\|_{H,*}^2\right]\\
    \stackrel{\eqref{eq:lem6eq1}}{\leq} & \frac{1}{B} \mathbbm{E}\left[4\frac{L_i}{m^2p_i^2}(\phi_i(\widetilde{v}^s) - \phi_i(v_{k+1}) - \langle \nabla\phi_i(v_{k+1}),\widetilde{v}^s-v_{k+1}\rangle) + 2\|\nabla\phi(\widetilde{v}^s) - \nabla \phi(v_{k+1})\|_{H,*}^2\right]\\
    =& \frac{1}{B}[4\bar{L}(\phi(\widetilde{v}^s) - \phi(v_{k+1}) - \langle \nabla\phi(v_{k+1}),\widetilde{v}^s-v_{k+1}\rangle)+ 2\|\nabla\phi(\widetilde{v}^s) - \nabla \phi(v_{k+1})\|_{H,*}^2]\\
    \stackrel{\eqref{eq:lem6eq1}}{\leq}& \frac{8\bar{L}}{B}(\phi(\widetilde{v}^s) - \phi(v_{k+1}) - \langle \nabla\phi(v_{k+1}),\widetilde{v}^s-v_{k+1}\rangle)
\end{align*}
\end{proof}

\begin{lemma}[Coupling step 1, batch version]\label{lem:7}
Consider one inner loop of Algorithm \ref{algo:PDSMD-B}, where the randomness only comes from the choice of $I$. It satisfies that for $\forall u$:
\begin{align*}
    &\alpha_s \langle \nabla \phi(v_{k+1}), z_k - u\rangle\\
\leq &\frac{\alpha_s}{\tau_{1,s}} \left\{\phi(v_{k+1}) - \mathbbm{E}[\phi(y_{k+1})] + \tau_2 \phi(\widetilde{v}^s) - \tau_2 \phi(v_{k+1}) - \tau_2 \langle\nabla\phi(v_{k+1}),\widetilde{v}^s - v_{k-1}\rangle\right\} \\
& + V_{z_k}(u) - \mathbbm{E}[V_{z_{k+1}}(u)].
\end{align*}
\end{lemma}
\begin{proof}
One can easily check that the Lemma E.1. and Lemma E.3. in \cite{allen2017katyusha} holds for the batch version of PDASMD, where $\psi(\cdot) = 0$ in these two lemmas for our case. 
Then we have
\begin{align}
\begin{split}\label{eq:lem7eq1}
    \phi(v_{k+1}) - \mathbbm{E}[\phi(y_{k+1})]\geq &  \mathbbm{E}\left[-\min_{y}\left\{\frac{9\bar{L}}{2}\|y - v_{k+1}\|_H^2 + \langle\widetilde{\nabla}_{k+1},y - v_{k+1}\rangle \right\}\right] \\
    &\quad - \frac{1}{16\bar{L}}\mathbbm{E}[\|\widetilde{\nabla}_{k+1} - \nabla \phi(v_{k+1})\|_{H,*}^2].
\end{split}
\end{align}
\begin{equation}\label{eq:lem7eq2}
    \alpha_s \langle\widetilde{\nabla}_{k+1},z_{k+1} - u\rangle \leq - \frac{1}{2}\|z_k - z_{k+1}\|_H^2 + V_{z_k}(u) - V_{z_{k+1}}(u). 
\end{equation}
Then 
\begin{align}
    &\alpha_s \langle \widetilde{\nabla}_{k+1}, z_k - u\rangle = \alpha_s \langle \widetilde{\nabla}_{k+1}, z_k - z_{k+1}\rangle + \alpha_s \langle \widetilde{\nabla}_{k+1}, z_{k+1} - u\rangle\nonumber\\
    \stackrel{\eqref{eq:lem7eq2}}{\leq} & \alpha_s \langle \widetilde{\nabla}_{k+1}, z_k - z_{k+1}\rangle - \frac{1}{2}\|z_k - z_{k+1}\|_H^2 + V_{z_k}(u) - V_{z_{k+1}}(u).\label{eq:lem7eq3}
\end{align}
To bound $\alpha_s \langle \widetilde{\nabla}_{k+1}, z_k - z_{k+1}\rangle - \frac{1}{2}\|z_k - z_{k+1}\|_H^2$, consider the variable $v:= \tau_{1,s}z_{k+1} + \tau_2 \widetilde{v}^s + (1-\tau_{1,s} - \tau_2)y_k$, then $v_{k+1} - v = \tau_{1,s}(z_k - z_{k+1})$. 
We have that 
\begin{align}
    &\mathbbm{E}\left[\alpha_s \langle \widetilde{\nabla}_{k+1}, z_k - z_{k+1}\rangle - \frac{1}{2}\|z_k - z_{k+1}\|_H^2\right] \nonumber\\
    =& \mathbbm{E}\left[\frac{\alpha_s}{\tau_{1,s}} \langle \widetilde{\nabla}_{k+1}, v_{k+1} - v\rangle - \frac{1}{2\tau_{1,s}^2}\|v_{k+1} - v\|_H^2\right] \nonumber\\
    =& \mathbbm{E}\left[\frac{\alpha_s}{\tau_{1,s}} \left(\langle \widetilde{\nabla}_{k+1}, v_{k+1} - v\rangle - \frac{9\bar{L}}{2}\|v_{k+1} - v\|_H^2\right)\right]\nonumber\\
    \stackrel{\eqref{eq:lem7eq1}}{\leq} &  \frac{\alpha_s}{\tau_{1,s}}\left(\phi(v_{k+1}) - \mathbbm{E}[\phi(y_{k+1})] + \frac{1}{16\bar{L}}\mathbbm{E}[\|\widetilde{\nabla}_{k+1} - \nabla \phi(v_{k+1})\|_{H,*}^2]\right) \nonumber\\
    \stackrel{\eqref{eq:lem6}}{\leq}& \frac{\alpha_s}{\tau_{1,s}}\left(\phi(v_{k+1}) - \mathbbm{E}[\phi(y_{k+1})] + \frac{1}{2B}(\phi(\widetilde{v}^s) - \phi(v_{k+1}) - \langle \nabla\phi(v_{k+1}),\widetilde{v}^s-v_{k+1}\rangle)\right).\label{eq:lem7eq4}
\end{align}
Take expectation on both sides of inequality \eqref{eq:lem7eq3}, plug in inequality \eqref{eq:lem7eq4} and notice that $\mathbbm{E}[\langle \widetilde{\nabla}_{k+1}, z_k - u\rangle] = \langle  \nabla \phi(v_{k+1}), z_k - u\rangle$ and $\tau_2 = \frac{1}{2B}$, we get the desired bound.
\end{proof}

The rest of the proof for Theorem \ref{thm:convergence_batch} is simply repeating the steps in Appendix \ref{app:A}, except that we replace Lemma \ref{lem:01} with Lemma \ref{lem:7}. So we omit the details of the proof.

\section{Details of Numerical Study}\label{app:G}
\paragraph{Data description.} We use both synthetic and real grey-scale images as the marginal distribution. 
For the simulated data, we follow the data generation mechanism in \cite{altschuler2017near,xie2022}. 
The images are generated by randomly positioning a square foreground on a background, with the foreground occupying about $20\%$ of the space. 
The foreground has each pixel value randomly drawn from uniform $[0,3]$, and the background has each pixel value randomly drawn from uniform $[0,1]$. 
Figure \ref{fig:0} shows some examples of the generated images.  

\begin{wrapfigure}{r}{0.35\textwidth}
  \centering
  \vspace*{-10mm}
    \includegraphics[width=\linewidth]{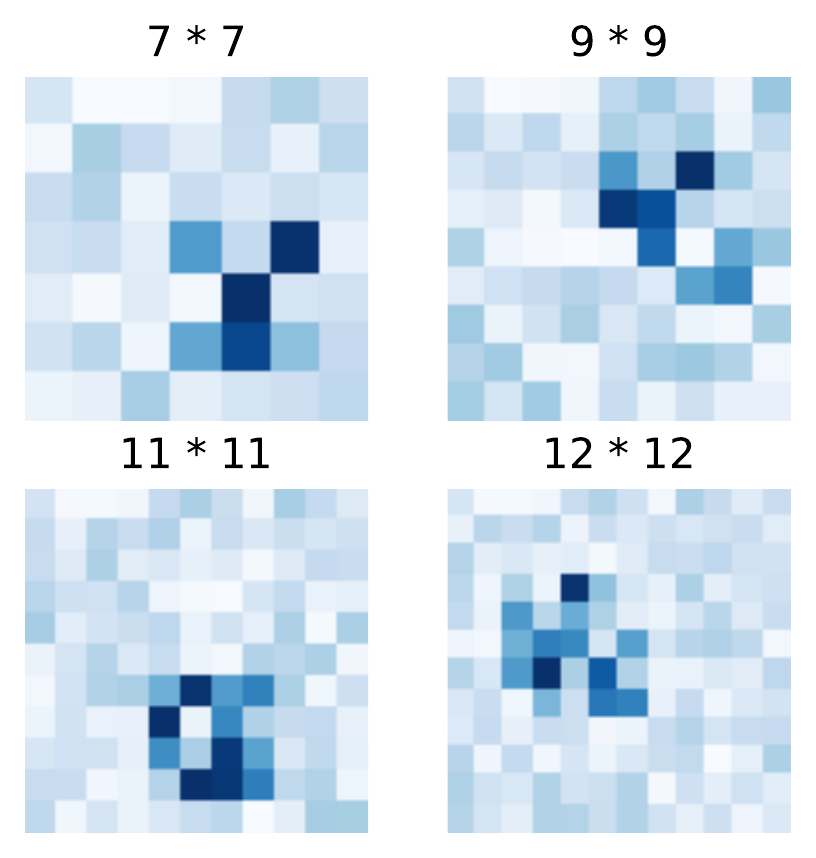}
    \vspace*{-5mm}
  \caption{Synthetic image example.}
  \vspace*{-8mm}
  \label{fig:0}
\end{wrapfigure}

For the real data, we randomly sample from the hand-written MNIST data set. 
Then we downscale the images to adjust the size of the marginal distribution.
We also add a background with a relatively small intensity to the down-scaled images to avoid numerical issues. 
With the marginal distribution determined, the cost matrix has each element calculated as the $l_1$ distance between the pixel locations on the image. 

\paragraph{Algorithm implementation.}
We compare the computational efficiency of the algorithms by measuring the number of arithmetic operations they use for finding an $\epsilon$-solution of the OT between two marginal distributions for a fixed $\epsilon$. 
To achieve this, all the algorithms are run with a rounding step. 
Thus for PDASMD, we run Algorithm \ref{algo:2}. 
In particular, for step 2 of Algorithm \ref{algo:2}, the PDASMD algorithm is run with the number of inner loops set to the problem size, $w(\cdot) = \frac{1}{2}\|\cdot\|_2^2$ and $\|\cdot\|_H = \|\cdot\|_{\infty}$. 
We run the PDASGD algorithm by changing $\|\cdot\|_H$ to $\|\cdot\|_2$ compared to the PDASMD. 
Note that the PDASGD algorithm is essentially equivalent to that of \cite{xie2022}.
We also run APDAGD \cite{dvurechensky2018computational}, AAM \cite{guminov2021combination}, Sinkhorn \cite{dvurechensky2018computational}, APDRCD \cite{guo2020fast} and Stochastic Sinkhorn (Algorithm \ref{algo:4}) for comparison. 
The implementation of all the algorithms above follows their standard definitions; there is no hyper-parameter to tune.

We also implement experiments for PDASMD-B on both synthetic and real data. 
For a fixed pair of marginals, PDASMD-B is implemented for a sequence of batch sizes. 
The number of inner loops is set to be the problem size divided by the batch size, which matches the setting in Corollary \ref{cor:computation_ot_b}, and we take $w(\cdot) = \frac{1}{2}\|\cdot\|_2^2$ and $\|\cdot\|_H = \|\cdot\|_{\infty}$, which are the same as the experiment of PDASMD.
For each batch size, the total number of computations and the running time are recorded. 

All our experiments are run on Google Colab using \textbf{NO} GPU or TPU accelerator. 
%We attach the full code and data to reproduce our results in the Supplemental Material.

\section{Application of the PDASMD Algorithm to Machine Learning Tasks}\label{app:H}

Optimal Transport can be applied to modern machine-learning tasks such as domain adaptation and color transfer. 
In this section, we illustrate that our PDASMD algorithm, when applied to OT, can solve those problems. 

\paragraph{Domain Adaptation.} 
This experiment aims to show that our PDASMD algorithm, when applied to OT, can successfully perform domain adaptation.  
In short, domain adaptation means transferring knowledge from a source domain to a target domain for which data have different probability density functions. 
For more details on the domain adaptation problem description and its OT formulation, see \cite{courty2015}.

We use the two-moons example to illustrate the application of our PDASMD algorithm on domain adaptation. 
The two moons example uses simulated data. 
The source domain consists of two entangled moons, where each moon represents one class. 
The target domain is built by applying a rotation to the two moons. 
We sample $150$ labeled data points from each moon as our source domain. 
The target domain consists of the same number of samples, where the samples are independent of the source domain and are unlabeled. 
We use the labeled source domain data, transfer them to the target domain by OT using our PDASMD algorithm, and learn an SVM classifier with the Gaussian kernel using the transferred source data on the target domain. 
We test the generalization performance on 2,000 samples that follow the same distribution as the target domain. 

Figure \ref{fig:two_moon} shows the domain adaptation result. 
In Figure \ref{fig:two_moon}, we plot the source domain, target domain (for different rotation angles), the transformed density, and decision boundaries. 
From the plots, we see that the transformed density reasonably fits the major parts of the target domain when the rotation angle is not too large ($\leq 50^{\circ}$). 
This shows that the PDASMD algorithm successfully performs the domain adaptation.

We report the generalization performance of the domain adaptation in Table \ref{tab:02}. 
We have three columns: the rotation degree of the target domain in the two moons example, the mean classification error when the domain adaptation is performed using our PDASMD algorithm, and the mean classification error of OT-IT in \cite{courty2015} (where they solve entropic OT by the Sinkhorn algorithm). 
From Table \ref{tab:02}, we see that our PDASMD algorithm performs better than the Sinkhorn when the rotation degree is large ($> 50^{\circ}$).

\begin{table*}
[ht]
\caption{Mean Classification Error over $10$ Repetitions of the Two Moons Example.}
\label{tab:02}
\begin{center}
\begin{small}
\begin{sc}
%\vskip -0.15in
\begin{tabular}{ccc}
\toprule
Rotation Degree &  PDASMD - Classification Error & OT-IT  \cite{courty2015} Classification Error\\
\midrule
$10^\circ$ & 0.022& \textbf{0}\\
$20^\circ$ & 0.054& \textbf{0.007}\\
$30^\circ$ & \textbf{0.043}& 0.054\\
$40^\circ$ & 0.169& \textbf{0.102}\\
$50^\circ$ & \textbf{0.221}& \textbf{0.221}\\
$70^\circ$ & \textbf{0.317}& 0.398\\
$90^\circ$ & \textbf{0.488}& 0.508\\
\bottomrule
\end{tabular}
\end{sc}
\end{small}
\end{center}
\vskip -0.1in

\end{table*}

\begin{figure*}
%\vskip -0.2in
\begin{center}
    %\subfigure[rotation = $10^\circ$]{\includegraphics[width=1.4in]{images/OT domain adaptation degree 10.pdf}\label{fig:2a}} 
    \subfigure[rotation = $20^\circ$]{\includegraphics[width=1.4in]{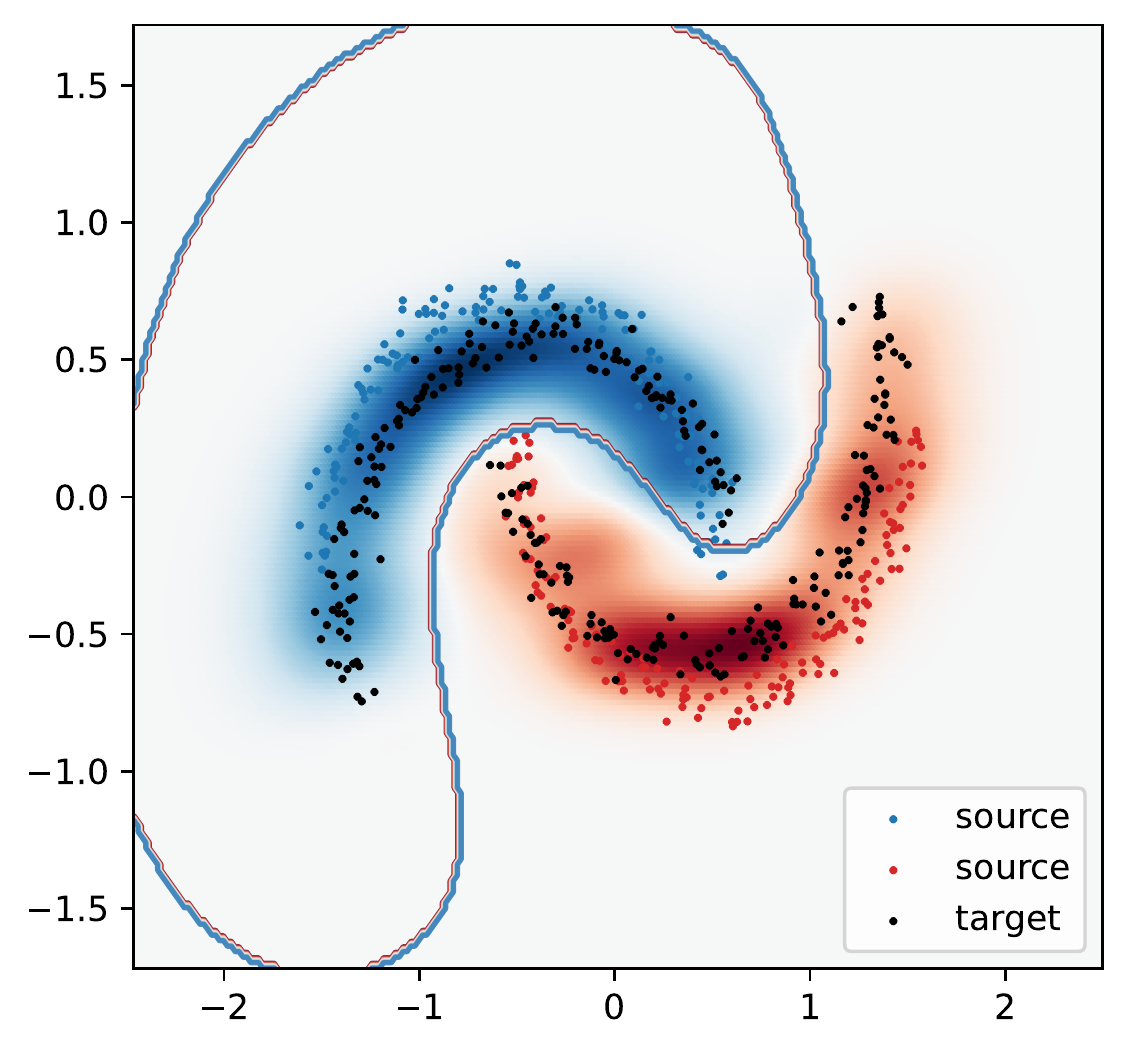}\label{fig:2b}} 
    \subfigure[rotation = $40^\circ$]{\includegraphics[width=1.4in]{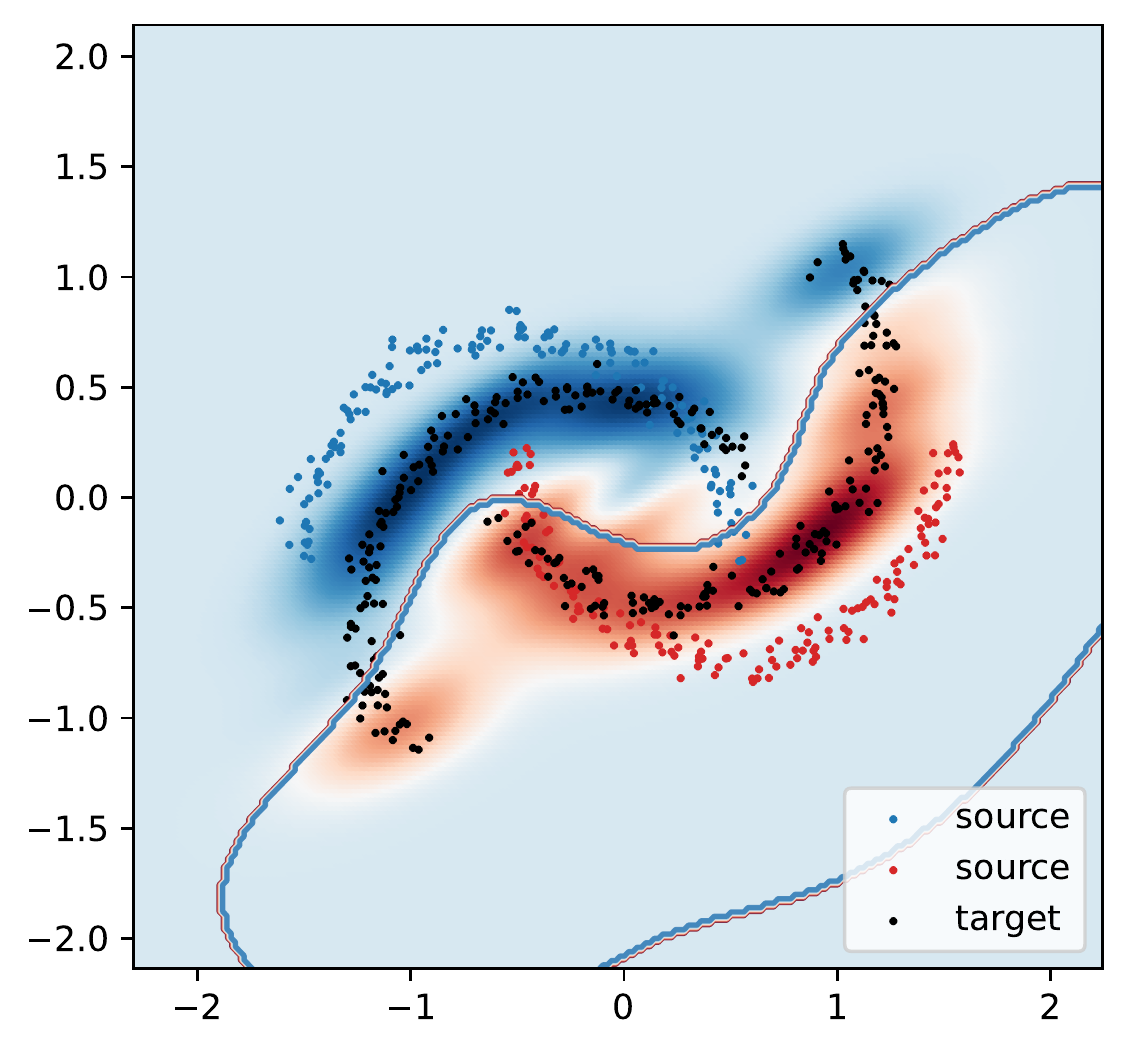}\label{fig:2c}} 
    \subfigure[rotation = $50^\circ$]{\includegraphics[width=1.4in]{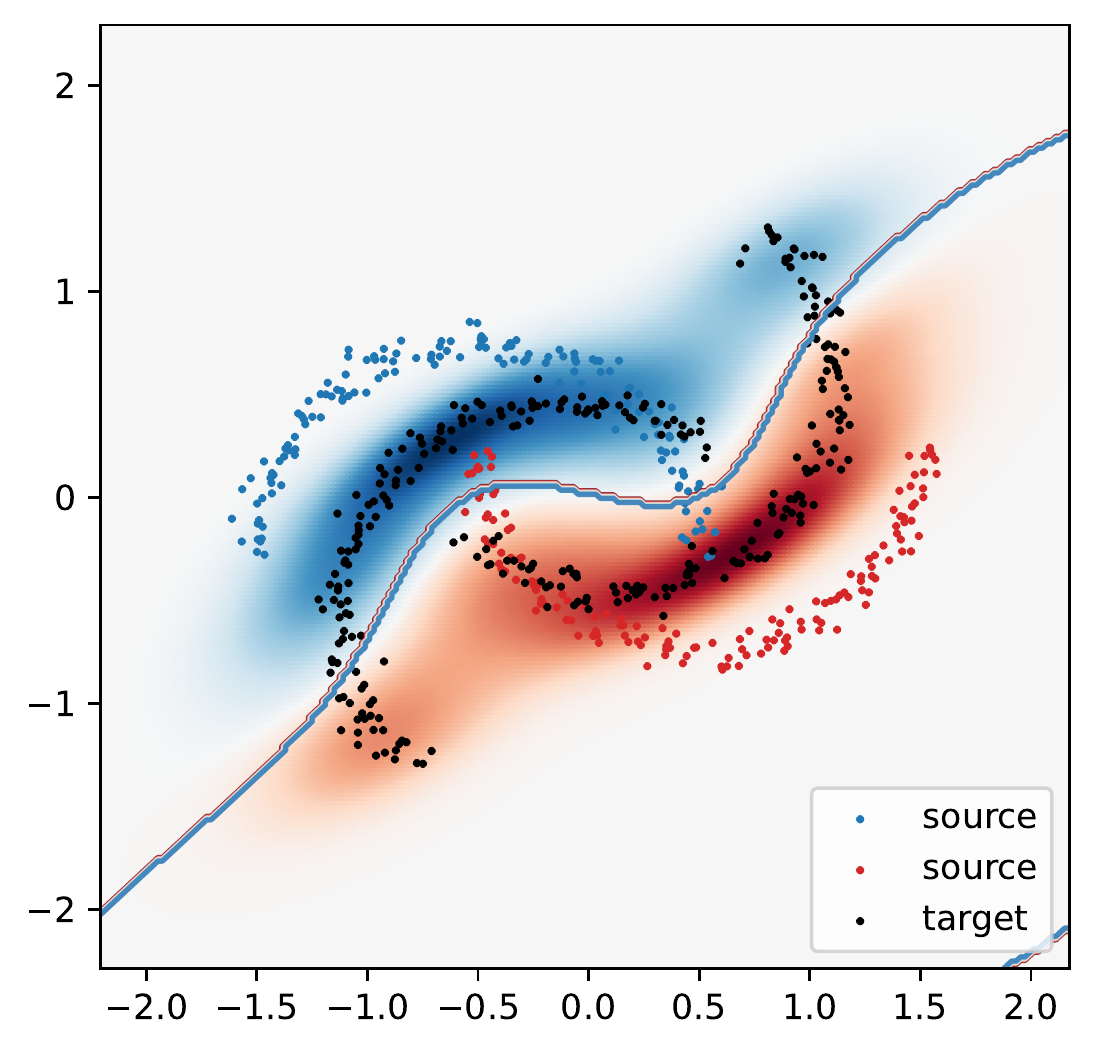}\label{fig:12}} 
    \subfigure[rotation = $90^\circ$]{\includegraphics[width=1.4in]{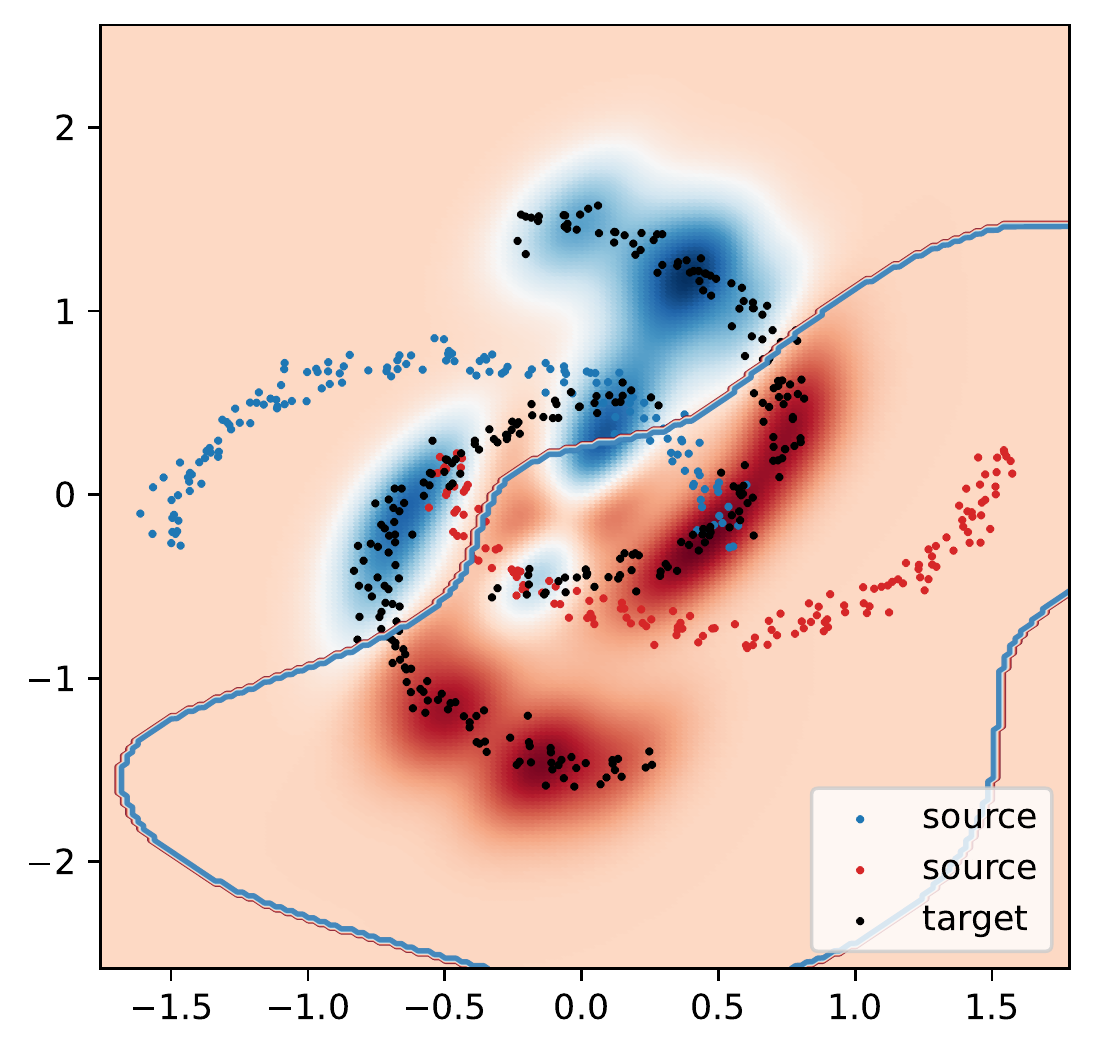}\label{fig:2e}} 
\caption{Two Moons Example}
\label{fig:two_moon}
\end{center}    
%\vskip -18pt    
\end{figure*}

\begin{figure}[H]
\vskip -0.2in
\begin{center}
    \includegraphics[width=5in]{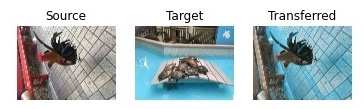} 
\caption{Color Transfer Example}
\label{fig:col_trans}
\end{center}    
\vskip -18pt    
\end{figure}

\paragraph{Color Transfer.} 
This experiment shows that our PDASMD algorithm successfully performs the color transfer task. 
The color transfer takes two input images and imposes the color palette of the first image onto a second one. 
Color transfer can be formulated as an OT problem. 
For more details see references \cite{ferradans2014regularized,rabin2014adaptive}. 

For an example of the color transfer problem, we apply our PDASMD algorithm to solve the corresponding OT problem. 
We show the color transfer results in Figure \ref{fig:col_trans}. 
Though we cannot evaluate the color transfer result quantitatively, one can tell from Figure \ref{fig:col_trans} that the color of the target image has been successfully transferred to the source image. 
This shows that our PDASMD algorithm successfully performs the color transfer task. 

\end{document}